\documentclass[oneside,leqno,10pt]{article}
\usepackage{amssymb,amsmath,latexsym,amsthm,stmaryrd}

\def\ga{\mathfrak{a}}

\def\ge{\mathfrak{e}}
\def\gf{\mathfrak{f}}
\def\gg{\mathfrak{g}}
\def\gh{\mathfrak{h}}

\def\gk{\mathfrak{k}}
\def\gl{\mathfrak{l}}
\def\gm{\mathfrak{m}}
\def\gn{\mathfrak{n}}
\def\go{\mathfrak{o}}
\def\gp{\mathfrak{p}}

\def\gs{\mathfrak{s}}
\def\gt{\mathfrak{t}}
\def\gu{\mathfrak{u}}
\def\gv{\mathfrak{v}}

\def\gz{\mathfrak{z}}

\def\C{\mathbb{C}}

\def\H{\mathbb{H}}

\def\O{\mathbb{O}}

\def\Q{\mathbb{Q}}
\def\R{\mathbb{R}}

\def\Z{\mathbb{Z}}

\def\cC{\mathcal{C}}

\def\cH{\mathcal{H}}

\def\cO{\mathcal{O}}

\def\Im{{\rm Im}\,}

\def\Span{{\rm Span}\,}
\def\Ad{{\rm Ad}\,}
\def\ad{{\rm ad}\,}
\def\Pf{{\rm Pf}\,}
\def\rank{{\rm rank}\,}
\def\Ind{{\rm Ind\,}}
\def\tr{{\rm trace\,}}

\newtheorem{theorem}[equation]{Theorem}

\newtheorem{lemma}[equation]{Lemma}

\newtheorem{proposition}[equation]{Proposition}

\newtheorem{definition}[equation]{Definition}

\def\sideremark#1{\ifvmode\leavevmode\fi\vadjust{\vbox to0pt{\vss
 \hbox to 0pt{\hskip\hsize\hskip1em
\vbox{\hsize2cm\tiny\raggedright\pretolerance10000 
 \noindent #1\hfill}\hss}\vbox to8pt{\vfil}\vss}}} 

\title{Plancherel Formulae associated to Filtrations of Nilpotent Lie Groups}

\author{Joseph A. Wolf}

\begin{document}

\maketitle

\abstract{We study the conditions for a nilpotent Lie group to be foliated
into subgroups that have square integrable (relative discrete series)
unitary representations, that fit together to form a filtration by normal
subgroups.  Then we use that filtration
to construct a class of ``stepwise square integrable'' representations on
which Plancherel measure is concentrated.  Further,
we work out the character formulae for those stepwise square integrable
representations, and we give an explicit Plancherel formula.  Next, we
use some structure theory to check that all these constructions and results
apply to nilradicals of minimal parabolic subgroups of real reductive Lie
groups.  Finally, we develop multiplicity formulae for compact quotients
$N/\Gamma$ where $\Gamma$ respects the filtration.}

\section{Introduction}
There is a well developed theory of square integrable representations
of nilpotent Lie groups \cite{MW1973}.  It is, of course, based on
the general representation theory \cite{K1962} for nilpotent Lie
groups.  A connected simply connected Lie group $N$
with center $Z$ is called {\em square integrable} if it has unitary
representations $\pi$ whose coefficients $f_{u,v}(x) = 
\langle u, \pi(x)v\rangle$ satisfy $|f_{u,v}| \in L^2(N/Z)$.  If $N$ has one 
such square integrable representation then there is a certain polynomial
function $\Pf(\lambda)$ on the linear dual space $\gz^*$ of the Lie algebra of
$Z$ that is key to harmonic analysis on $N$.  Here $\Pf(\lambda)$ is the
Pfaffian of the antisymmetric bilinear form on $\gn / \gz$ given by
$b_\lambda(x,y) = \lambda([x,y])$.  The square integrable
representations of $N$ are certain easily--constructed representations
$\pi_\lambda$ where $\lambda \in \gz^*$ with $\Pf(\lambda) \ne 0$,
Plancherel  almost irreducible unitary representations of $N$ are square
integrable, and up to an explicit constant 
$|\Pf(\lambda)|$ is the Plancherel density of the unitary
dual $\widehat{N}$ at $\pi_\lambda$.  
\medskip

This theory has proved to have serious analytic consequences \cite{W2007}.  
\medskip

In this paper we present an extension of that theory.  Under certain
conditions, the nilpotent Lie group $N$ has a particular decomposition into 
subgroups that have square integrable representations, and the Plancherel 
formula then is synthesized explicitly in terms of the Plancherel formulae
of those subgroups.  Many of our calculations are on the Lie algebra level,
reflecting the setting of the square integrability conditions mentioned above.
\medskip

The guiding example for these decompositions is that of the upper
triangular real matrices.  We go through it separately in Section 
\ref{up-triang-decomp} for the decompositions, Section \ref{up-triang-l2}
for the square integrability, and Section \ref{up-triang-planch}
for the Plancherel formula.  We separate out the upper triangular
groups for two reasons.  First, they give a clear illustration of the 
technical conditions that we need  more generally to decompose the group
and to reconstitute the Plancherel formula.  Second, and equally important, 
they appear in many situations and are of independent interest.  
The precise conditions are given in (\ref{setup}).
\smallskip

The decompositions for upper triangular matrices 
were suggested by the work of Mar\'\i a L. Barberis and Isabel G. Dotti
on abelian complex structures; see \cite{BD2001}, \cite{BD2004} and 
\cite{Ka}.
I thank them for discussions of their work in progress \cite{BD2011}
on those structures.
\medskip

The general formulation and the resulting Plancherel formulae are the 
content of Section \ref{general_theory}.  There we extend
the material of Sections \ref{up-triang-decomp}, \ref{up-triang-l2} 
and \ref{up-triang-planch} to a much more general setting.  
\medskip

In Section \ref{iwasawa} we verify the conditions of Section 
\ref{general_theory} for the unipotent radicals of parabolic subgroups of
real semisimple Lie groups.  This has many potential applications in
differential geometry, in hypoelliptic differential equations, and in
harmonic analysis.  For example, in the case of cuspidal parabolics, it
has the potential of simplifying some of the integrations
in Harish--Chandra's theory of the constant term.  These unipotent
radical examples also appear in many other geometric and analytic settings.
\medskip

Finally, in Section \ref{quotients} we consider the case where  our connected
simply connected nilpotent Lie group $N$ has a discrete co-compact
subgroup $\Gamma$ that fits into the pattern of Section \ref{general_theory}.
We show that the compact nilmanifold $N/\Gamma$ has a
corresponding foliation and derive analytic results analogous to those
of Theorem \ref{plancherel-general}.
These results include multiplicity formulae for stepwise square integrable
representations as summands of the regular representation 
$\Ind_\Gamma^N(1_\Gamma)$
of $N$ on $L^2(N/\Gamma)$. They apply in particular to the nilradicals
of minimal parabolic subgroups, as studied in Section \ref{iwasawa}.
\medskip

Section \ref{appendix} is an appendix demonstrating the root structure of
the key technical tool, Lemma \ref{layers-nilpotent}, behind
Theorem \ref{iwasawa-layers}.  I expect that it will only appear in the
arXiv version of this paper.

\section{Decomposition of Upper Triangular Matrices}
\label{up-triang-decomp}
\setcounter{equation}{0}

Let $\gn$ denote the real Lie algebra of 
$\ell \times \ell$ matrices with zeroes
on and below the diagonal, and let $N$ the corresponding 
unipotent Lie group of $\ell \times \ell$ matrices with zeroes below the 
diagonal and ones on the diagonal.
\medskip

As usual $e_{i,j} \in \gn$ denotes the matrix with $1$ in row $i$ column $j$ 
and zeroes elsewhere, so $\gn$ is the span of 
$\{e_{i,j} \mid 1 \leqq i < j \leqq \ell\}$.
Here $[e_{i,j}, e_{m,n}]$ is $e_{i,n}$ if $i < j = m < n$, $-e_{m,j}$ if
$m < n = i < j$, $0$ in all other cases.  Thus, for 
$1 \leqq r \leqq [\tfrac{\ell}{2}]$, we define
\begin{equation}\label{2.1}
\begin{aligned}
\gm_r :=& \Span\{e_{r,s} \mid r+1 \leqq s \leqq \ell - r\} 
	\cup \Span\{e_{q,\ell - r + 1} \mid r+1 \leqq q \leqq \ell - r \}
	\cup e_{r,\ell - r + 1}\R \text{ and } \\
\gn_r :=& \gm_1+ \gm_{2} + \dots + \gm_r
       = \Span \{e_{i,j} \mid i \leqq r \leqq \ell - r + 1 < j\}.
\end{aligned}
\end{equation}
Then $\gm_r$ is a subalgebra of $\gn$ that is isomorphic to the Heisenberg 
algebra $\gh_{\ell - 2r}$ of dimension $2(\ell - 2r) + 1$; it has center 
$\gz_r := e_{r,\ell - r + 1}\R$.  Note that
\begin{equation}\label{mm}
[\gm_i , \gm_j] \subset \gn_{min(i,j)}\,, \text{ so each }
\gn_r \text{ is an ideal in } \gn\,.
\end{equation}
In particular we have semidirect sum decompositions
\begin{equation}\label{2.3}
\gn_r = \gn_{r-1} \subsetplus \gm_r
\end{equation}
and a filtration
\begin{equation}\label{2.4}
\gn_1 \subset \gn_2 \subset \dots \subset \gn_{[\ell / 2]} = \gn
\end{equation}
by ideals.
\medskip

Fix the positive definite inner product on $\gn$ in which the $e_{i,j}$ are
orthonormal.  Define $J_r : \gn \to \gn$ by 
$\langle J_r(x), y \rangle = \langle e_{r,\ell - r + 1}, [x,y]\rangle$.  Its
image is the non-central part 
\begin{equation}
\Span\{e_{r,s} \mid r+1 \leqq s \leqq \ell - r\}
        \cup \Span\{e_{q,\ell - r + 1} \mid r+1 \leqq q \leqq \ell - r \}
        = \gm_r \cap \gv_r
\end{equation} 
of $\gm_r$ and its kernel is the central part $\gz_r = e_{r,\ell - r + 1}\R$.
That is the connection with abelian complex structures mentioned in the
Introduction.
\medskip

Group structure here follows algebra structure immediately, as all the 
groups are
unipotent and thus equal to the exponential image of their Lie algebras.
Thus we have closed connected subgroups $M_r \cong H_{\ell - 2r}$ and
normal closed connected subgroups
$N_r = M_1 M_2 \dots M_r$ in $N$, and semidirect product decompositions
$N_r = N_{r-1} \rtimes M_r$.

\section{Square Integrability for Upper Triangular Matrices}
\label{up-triang-l2}
\setcounter{equation}{0}
Now we cascade down antidiagonally from the upper right hand
corner. For convenience let 
$m := [\tfrac{\ell}{2}]$.  If $r < m$ them $M_r$ is the Heisenberg
group of dimension $2(\ell - 2r) + 1$.
If $\ell$ is odd then $M_m$ is the $3$--dimensional Heisenberg
group, and if $\ell$ is even then $M_m \cong \R$ is the $1$--dimensional vector
group.  The point is that Plancherel-almost-every irreducible unitary 
representation of $M_r$ is a representation $\pi_{\lambda_r}$
specified by a nonzero linear functional $\lambda_r$ on the
center $\gz_r$ of $\gm_r$\,,
and that representation has matrix coefficients in $L^2(M_r/Z_r)$.
Write $\cH_{\pi_{\lambda_r}}$ for the representation space, or just
$\cH_r$ if there is no chance of confusion.
\medskip

Let $\lambda = \lambda_1 + \dots + \lambda_m$ where $0 \ne \lambda_r \in \gz_r$
for $1 \leqq r \leqq m$.  We are going to put together the square
integrable representations $\pi_{\lambda_r} \in \widehat{M_r}$ to form a
representation $\pi_r \in \widehat{N}$.  This will be a recursion on $r$
and we will need the
$$
S_r = Z_1 Z_2 \dots Z_r = S_{r-1} \times Z_r
$$
for that recursive construction.
\medskip
\begin{lemma}\label{norm}
$M_r$ centralizes $S_{r-1}$.
\end{lemma}
\begin{proof} The Lie algebra $\gs_{r-1}$ is spanned by the $e_{j,\ell + 1 -j}$
for $j = 1, \dots , r-1$.  Let $e_{u,v} \in \gm_r$.  Then
$[e_{u,v},e_{j,\ell + 1 -j}] = 0$ because $v > j$ and $u < \ell + 1 -j$.
\end{proof}

Express $N_2$ as the semidirect product $N_1\rtimes M_2$.  
Plancherel-almost-every irreducible unitary representation of $N_1 = M_1$
is a representation $\pi_{\lambda_1}$ specified by a nonzero linear functional 
$\lambda_1 \in \gz_1^*$.  View $\lambda_1$ as an element of $\gn_1^*$
that vanishes on the non-central matrices $e_{i,j}$ in $\gn_1$.
Choose an invariant polarization $\gp_1' \subset \gn_2$ for the linear
functional $\lambda_1' \in \gn_2^*$ that agrees with $\lambda_1$ on $\gn_1$
and vanishes on $\gm_2$.  Lemma \ref{norm} implies 
$\ad^*(\gm_2)(\lambda_1')|_{\gz_1 + \gm_2} = 0$,
so  $\gp_1' = \gp_1 + \gm_2$ where $\gp_1$ is an 
invariant polarization for the linear functional $\lambda_1 \in \gn_1^*$.
The associated representations
are $\pi_{\lambda_1'} \in \widehat{N_2}$ and $\pi_{\lambda_1} \in
\widehat{N_1}$.  Note that $N_2/P_1' = N_1/P_1$\,,
so the representation spaces $\cH_{\pi_{\lambda_1'}} = L^2(N_2/P_1')
= L^2(N_1/P_1) = \cH_{\pi_{\lambda_1}}$.
In other words, $\pi_{\lambda_1'}$ extends $\pi_{\lambda_1}$ to a
unitary representation of $N_2$ on the same Hilbert space
$\cH_{\pi_{\lambda_1}}$, and $d\pi_{\lambda_1'}(\gz_2) = 0$.
Now the Mackey Little Group method gives us

\begin{lemma}\label{ext-1-2}
The irreducible unitary representations of $N_2$, whose restrictions to
$N_1$ are multiples of $\pi_{\lambda_1}$, are the 
$\pi_{\lambda_1'} \widehat\otimes\gamma$ where $\gamma \in \widehat{M_2}
= \widehat{N_2/N_1}$\,.
\end{lemma}

Given nonzero $\lambda_1 \in \gz_1^*$ and $\lambda_2 \in \gz_2^*$ we have 
representations $\pi_{\lambda_1} \in \widehat{M_1}$ and 
$\pi_{\lambda_2} \in \widehat{M_2}$ with coefficients in $L^2(M_1/Z_1)$
and $L^2(M_2/Z_2)$ respectively.  Using the notation of Lemma \ref{ext-1-2} 
we define
\begin{equation}
\pi_{\lambda_1 + \lambda_2} \in \widehat{N_2} \text{ by }
\pi_{\lambda_1 + \lambda_2} 
= \pi'_{\lambda_1} \widehat\otimes \pi_{\lambda_2}\,.
\end{equation}
We now use the square integrability of $\pi_{\lambda_1}$ and
$\pi_{\lambda_2}$ for some square integrability of 
$\pi_{\lambda_1 + \lambda_2}$\,.

\begin{proposition}\label{L2.1}
The coefficients $f_{z,w}(xy)= \langle z, 
\pi_{\lambda_1 + \lambda_2}(xy)w\rangle$ of $\pi_{\lambda_1 + \lambda_2}$ 
are in $L^2(N_2/S_2)$, in fact satisfy 
$
||f_{z,w}||^2_{L^2(N_r/S_r)} = \tfrac{||z||^2 ||w||^2}{\deg(\pi_{\lambda_1})
        \dots \deg(\pi_{\lambda_r})}\,.
$
\end{proposition}

\begin{proof}
We write $\cH_r$ for the representation space of $\pi_{\lambda_r}$.
$\cH_1$ also is the representation space for $\pi_{\lambda_1'}$\,, so
$\pi_{\lambda_1 + \lambda_2}$ has representation space
$\cH_1 \widehat\otimes\cH_2$.  Choose nonzero vectors $u, v \in \cH_1$
and $u', v' \in \cH_2$.  We need only prove that the function
$f(x,y) = \langle u,\pi'_{\lambda_1}(xy)v\rangle
\langle u', \pi_{\lambda_2}(y)v'\rangle$, $x \in M_1$ and $y \in M_2$\,, 
satisfies 
$||f||^2_{L^2(N_2/S_2)} = 
	\left ( \frac{||u||^2 ||v||^2}{\deg(\pi_{\lambda_1})} \right )
	\left ( \frac{||u'||^2 ||v'||^2}{\deg(\pi_{\lambda_2})} \right )$,
so that the coefficients
$xy \mapsto \langle u\otimes u', \pi_{\lambda_1 + \lambda_2}(xy)(v\otimes v')
\rangle$ of decomposable vectors are in $L^2(N_2/S_2)$.   For that,
let $\{z_i\}$ and $\{w_j\}$ be complete orthonormal sets in
$\cH_1 \widehat\otimes\cH_2$. Suppose that both $\sum |a_{i,j}|^2$ and
$\sum |b_{i,j}|^2$ are finite, so $z = \sum a_{i,j}z_i\otimes w_j$
and $w = \sum b_{i,j}z_i\otimes w_j$ are general elements of
$\cH_1 \widehat\otimes\cH_2$.  Then the coefficient
$\langle x, \pi_{\lambda_1 + \lambda_2}(xy)w\rangle
= \sum a_{i,j}\overline{b_{i',j'}} \langle z_i, 
	\pi_{\lambda_1 + \lambda_2}(xy)w_j\rangle$
has square $L^2(N_2/S_2)$--norm
$\frac{1}{\deg(\pi_{\lambda_1})}\frac{1}{\deg(\pi_{\lambda_2})}
\sum |a_{i,j}|^2 \sum |b_{i',j'}|^2 =
\frac{||z||^2||w||^2}{\deg(\pi_{\lambda_1})\deg(\pi_{\lambda_2})}
< \infty$.
\medskip

In order to integrate $|f|^2$ over $N_2 = M_1M_2$ modulo $S_2 = Z_1Z_2$
we use the fact that the action of $M_2$ on $\gm_1$ is unipotent, so there
is a measure preserving decomposition
\begin{equation}\label{E1}
N_2/S_2 = \left ( M_1/Z_1 \right )
	\times \left ( N_2/Z_2 \right ).
\end{equation}
Using the extension
of Schur Orthogonality to representations with coefficients that are
square integrable modulo the center of the group, and writing $v_y$ for 
$\pi_{\lambda'_1}(y)v$, we compute
$$
\begin{aligned}
||f||^2_{L^2(N_2/S_2)} &= 
	\int_{N_2/S_2} 
	  |\langle u,\pi'_{\lambda_1}(xy)v\rangle|^2
	  |\langle u', \pi_{\lambda_2}(y)v'\rangle|^2 
	  d(xyZ_1 Z_2) \\
&= \int_{M_2/Z_2} |\langle u', \pi_{\lambda_2}(y)v'\rangle|^2
   \left ( \int_{M_1/Z_1}
	|\langle u,\pi_{\lambda'_1}(xy)v\rangle|^2
	d(xZ_1) \right ) d(yZ_2) \\
&= \int_{M_2/Z_2} |\langle u', \pi_{\lambda_2}(y)v'\rangle|^2
   \left ( \int_{M_1/Z_1}
        |\langle u, \pi_{\lambda'_1}(x)v_y\rangle|^2
	d(xZ_1) \right ) d(yZ_2) \\
&= \int_{M_2/Z_2} |\langle u', \pi_{\lambda_2}(y)v'\rangle|^2
   \left ( \int_{M_1/Z_1}
	|\langle u, \pi_{\lambda_1}(x)v_y\rangle|^2
        d(xZ_1) \right ) d(yZ_2) \\
&= \int_{M_2/Z_2} |\langle u', \pi_{\lambda_2}(y)v'\rangle|^2
	\tfrac{||u||^2 ||v_y||^2}{\deg(\pi_{\lambda_1})} d(yZ_2)\\
&= \left ( \frac{||u||^2 ||v_y||^2}{\deg(\pi_{\lambda_1})} \right )
	\left ( \frac{||u'||^2 ||v'||^2}{\deg(\pi_{\lambda_2})} \right ) =
  \left ( \frac{||u||^2 ||v||^2}{\deg(\pi_{\lambda_1})} \right )
	\left ( \frac{||u'||^2 ||v'||^2}{\deg(\pi_{\lambda_2})} \right )\\
&= \frac{||u\otimes u'||^2 ||v \otimes v'||^2}{\deg(\pi_{\lambda_1})
	\deg(\pi_{\lambda_2})} < \infty.
\end{aligned}
$$
That completes the proof of Proposition \ref{L2.1}.
\end{proof}

Proposition \ref{L2.1} starts our recursive construction.  More
generally, $N_{r+1}$ is the semidirect product $N_r \rtimes M_{r+1}$.
We fix nonzero $\lambda_i \in \gz_i^*$ for $1 \leqq i \leqq r+1$,
and we start with the representation $\pi_{\lambda_1 + \dots + \lambda_r}$
constructed step by step from the square integrable representations
$\pi_{\lambda_i} \in \widehat{M_i}$ for $1 \leqq i \leqq r$.  The
representation space $\cH_{\pi_{\lambda_1 + \dots + \lambda_r}} =
\cH_{\pi_{\lambda_1}} \widehat\otimes \dots \cH_{\pi_{\lambda_r}}$.  The
coefficients of $\pi_{\lambda_1 + \dots + \lambda_r}$ have absolute
value in $L^2(N_r/S_r)$.  In fact they satisfy
$$
||f_{z,w}||^2_{L^2(N_r/S_r)} = \tfrac{||z||^2 ||w||^2}{\deg(\pi_{\lambda_1})
        \dots \deg(\pi_{\lambda_r})}\,.
$$
Then $\pi_{\lambda_1 + \dots + \lambda_r}$ extends to a representation
$\pi'{\lambda_1 + \dots + \lambda_r}$ of $L_{r+1}$ on the same Hilbert
space $\cH_{\pi_{\lambda_1 + \dots + \lambda_r}}$, and it satisfies
$d\pi'_{\lambda_1 + \dots + \lambda_r}(\gz_{r+1}) = 0$.  Thus, as in
Lemma \ref{ext-1-2},

\begin{lemma}\label{ext-r}
The irreducible unitary representations of $N_{r+1}$, whose restrictions to
$N_r$ are multiples of $\pi_{\lambda_1 + \dots + \lambda_r}$, are the
$\pi'_{\lambda_1 + \dots + \lambda_r} \widehat\otimes\gamma$ where 
$\gamma \in \widehat{M_{r+1}} = \widehat{N_{r+1}/N_r}$\,.
\end{lemma}

Recall $0 \ne \lambda_{r+1} \in \gz_{r+1}^*$ and the square integrable
representation $\pi_{\lambda_{r+1}}$ of $M_{r+1} = L_{r+1}/L_r$\,.
Computing exactly as in Proposition \ref{L2.1}, we define
$\pi_{\lambda_1 + \dots + \lambda_{r+1}} = \pi'_{\lambda_1 + \dots + \lambda_r}
\widehat\otimes \pi_{\lambda_{r+1}}$ and conclude that

\begin{proposition}\label{L2.2}
The coefficients $f_{z,w}(x_1 \dots x_{r+1})= \langle z, 
\pi_{\lambda_1 + \dots + \lambda_{r+1}}(x_1x_2 \cdots x_{r+1})w\rangle$ of 
$\pi_{\lambda_1 + \dots + \lambda_{r+1}}$
are in $L^2(N_{r+1}/S_{r+1})$, in fact satisfy 
$
||f_{z,w}||^2_{L^2(N_{r+1}/S_{r+1})} = \tfrac{||z||^2 ||w||^2}
{\deg(\pi_{\lambda_1}) \dots \deg(\pi_{\lambda_{r+1}})}\,.  $
\end{proposition}

Since the $M_r$ are Heisenberg groups, except that $M_m$\,, the
last one, is $1$--dimensional abelian in case the size $\ell$ of
the matrices is even, we have $\deg \pi_{\lambda_r} = |\lambda_r|^{d_r}$
where $\dim M_r = 2d_r + 1$ and $d_r = \ell -2r$.
Proposition \ref{L2.2} is the recursion step for our construction, and 
the end case $r+1 = m$ is

\begin{theorem}\label{L2}
Let $0 \ne \lambda_r \in \gz_r^*$ for $1 \leqq r \leqq m$ and set
$\lambda = \lambda_1 + \dots + \lambda_m$. 
Denote $\deg(\pi_\lambda) = \deg(\pi_{\lambda_1}) \dots \deg(\pi_{\lambda_m})$.
Then the coefficients
$f_{z,w}(x) = \langle z, \pi_\lambda(z)w\rangle$ of the irreducible
unitary representation $\pi_\lambda$ on $N$ are in $L^2(N/S)$ and
satisfy
$||f_{z,w}||^2_{L^2(N/S)} = \tfrac{||z||^2 ||w||^2}{\deg(\pi_\lambda)}
= ||z||^2 ||w||^2 /\prod |\lambda_r|^{\ell - 2r}$.  
\end{theorem}
\begin{definition}\label{stepwise1}
{\rm The representations $\pi_\lambda$\,, constructed as just above,
are the {\it stepwise square integrable} representations of $N$ relative
to the decompositions (\ref{2.1}), (\ref{2.3}) and (\ref{2.4}).}
\hfill $\diamondsuit$
\end{definition}

\section{Plancherel Formula for Upper Triangular Matrices}
\label{up-triang-planch}
\setcounter{equation}{0}

The Plancherel measure for the group $M_r$ is 
$2^{d_r} d_r! |\lambda_r|^{d_r}d\lambda_r$
where $\dim M_r = 2d_r + 1$ and $d\lambda_r$ is Lebesgue measure on $\gz_r^*$.
If $f \in L^1(M_r)$
we have $\dot \pi_{\lambda_r}(f) = \int_{M_r} f(x_r)\pi_{\lambda_r}(x_r)dx_r$.
One version of the Plancherel formula for $M_r$ is
\begin{equation}\label{planch-hs}
||f||^2_{L^2(M_r)} = 2^{d_r} d_r! \int_{\gz_r^*} 
	||\dot \pi_{\lambda_r}(f)||_{HS}^2 |\lambda_r|^{d_r}d\lambda_r
	\hskip .5cm (d_r = \ell -2r)
\end{equation}
where $f \in L^1(M_r)\cap L^2(M_r)$ and $||\cdot ||_{HS}$ is Hilbert--Schmidt
norm, and another is
\begin{equation}\label{planch-trace}
f(x) =c_r\int_{\gz_r^*}\Theta_{\pi_{\lambda_r}}(r_xf)|\lambda_r|^{d_r}d\lambda_r
\end{equation}
where 
$\Theta_{\pi_{\lambda_r}}$ is the distribution character of $\pi_{\lambda_r}$,
given by $\Theta_{\pi_{\lambda_r}}(h) = \tr \pi_{\lambda_r}(h)$, 
$f \in C^\infty_c(M_r)$, $c_r = 2^{d_r} d_r!$, and 
$r_x$ is right translation of functions, $(r_xf)(g) =  f(gx)$.  As we will see
in a moment from the formula, the distribution $\Theta_{\pi_{\lambda_r}}$
is tempered, i.e. extends by continuity to the Schwartz space 
$\cC(M_r)$. 
\medskip 

To make this
explicit one needs the character formula for $\pi_{\lambda_r}$, i.e. the
formula for the tempered distribution $\Theta_{\pi_{\lambda_r}}$.  That is
given as follows.  Define $h_1 \in C^\infty_c(\gm_r)$ by $h_1(\xi) =
h(\exp(\xi))$.  The geometric tangent space of $\Ad^*(M_r)\lambda_r$ is 
the coadjoint orbit $\Ad^*(M_r)\lambda_r$ itself, the
affine hyperplane $\lambda_r + \gz_r^\perp$ in $\gm_r^*\cong (\gm_r/\gz_r)^*$.
We use Lebesgue measure $d\nu_r$ on $(\gm_r/\gz_r)^*$ normalized so that
Fourier transform is an isometry from $L^2(\gm_r/\gz_r)$ onto
$L^2(\gm_r/\gz_r)^*$, and we translate $d\nu_r$ to a measure 
$d\nu_{\lambda_r}$ on the orbit.  Then, from \cite{P1967} and
\cite{MW1973},
\begin{equation}\label{cf1}
\Theta_{\pi_{\lambda_r}}(h) = 
  c_r^{-1} |\lambda_r|^{-d_r} 
     \int_{\Ad^*(M_r)\lambda_r}\widehat{h_1}(\xi)d\nu_{\lambda_r}(\xi)
	\hskip .5cm (d_r = \ell -2r)
\end{equation}
where $\widehat{h_1}$ is the Fourier transform of $h_1$.
For all this see \cite[Theorem 6 and its proof]{MW1973}.
\medskip

In order to extend these results from one group $M_r$ to the upper
triangular group $N$ we first need 

\begin{proposition}\label{supp-planch}
Plancherel measure on $\widehat{N}$ is concentrated on 
$\{\pi_\lambda \mid \lambda = \lambda_1 + \dots + \lambda_m\,, \,\,
 0 \ne \lambda_r \in \gz_r^*\,\,\, \forall r\}$.
\end{proposition}

\begin{proof}
If $\zeta \in \gs^*$ then
$e^{2\pi\sqrt{-1}\zeta}: \xi \mapsto e^{2\pi\sqrt{-1}\zeta(\log \xi)}$ on $S$
is a unitary character on $S$.
Denote the induced representation
$\widetilde{\zeta} = \Ind_S^N(e^{2\pi\sqrt{-1}\zeta})$.
Induction by stages says that the left regular representation of $N$
is $\Ind_{\{1\}}^N(1) = \int_{\gs^*} \widetilde{\zeta} d\zeta$
where $d\zeta$ is Lebesgue measure on $\gs^*$.
Let 
\begin{equation}\label{regset}
\gt^* = \{\lambda  = \lambda_1 + \dots + \lambda_m\,, \,\,
 0 \ne \lambda_r \in \gz_r^*\,\,\, \forall r\}
\text{ and } 
P(\lambda) = \lambda_1^{d_1}\lambda_2^{d_2}\dots\lambda_m^{d_m}\,.
\end{equation}  
Since $P$ is not identically zero
we can ignore its zero set in the direct integral, so
left regular representation of $N$
is $\Ind_{\{1\}}^N(1) = \int_{\gt^*} \widetilde{\zeta} d\zeta$.
\end{proof}
Next, we break up the bilinear form $b_\lambda$\,.
\begin{lemma}\label{split-b}
Decompose each $\gm_r = \gz_r + \gv_r$ where $\gv_r$ is the span of the
$e_{i,j}$ in $\gm_r$ but not in $\gz_r$\,, and similarly $\gn = \gs + \gv$.  
If $\lambda \in \gt^*$ then the antisymmetric bilinear form $b_\lambda$ on 
$\gv$ is the direct sum  $b_{\lambda_1} \oplus \dots \oplus b_{\lambda_m}$
of nondegenerate bilinear forms on the $\gv_r$.  Equivalently,
if $r \ne t$ then $[\gm_r,\gm_t] \subset \gv$.
\end{lemma}
\begin{proof} The equivalence is clear from the definition of $b_\lambda$.
Let $\rho$ denote reflection on the antidiagonal.  Suppose that
$e_{i,j} \in \gm_r$ and $e_{a,b} \in \gm_t$ with $r \ne t$.  Then
$e_{i,j}e_{a,b}$ is on the antidiagonal if and only if $\rho(e_{i,j}) =
e_{a,b}$, and this happens if and only if $e_{a,b}e_{i,j}$ is on the
antidiagonal.  As $\rho(\gm_r) = \gm_r$ and $\rho(\gm_t) = \gm_t$,
it follows that $[\gm_r,\gm_t] \subset \gv$.
\end{proof}

Given $\lambda \in \gt^*$, the coadjoint orbit $\cO(\lambda) := \Ad^*(N)\lambda$
is just $\Ad(M_1)\lambda_1 \times \dots \times \Ad(M_m)\lambda_m$, and we
have the measure $d\nu_\lambda = d\nu_{\lambda_1} \times \dots \times
d\nu_{\lambda_m}$ on it.  Denote $c = c_1c_2\dots c_m = 2^{d_1 + \dots + d_m}
d_1! d_2! \dots d_m!$ where we recall $d_r = \tfrac{1}{2}(\dim\gm_r -1)
= \ell - 2r$.  Let $f \in C_c^\infty(N)$ (or more generally
$f \in \cC(N)$) and $\widehat{f_1}$ the classical
Fourier transform of the lift $f_1(\xi) = f(\exp(\xi))$ of $f$ to $\gn$.
We use Lebesgue measure $d\nu$ on $(\gm/\gs)^*$ normalized so that
Fourier transform is an isometry of $L^2(\gm/\gs)$ onto $L^2(\gm/\gs)^*$.
\medskip

Now, exactly as in (\ref{planch-trace}) and (\ref{cf1}) we combine
the result \cite[Theorem, p. 17]{P1967} of Puk\' anszky with the
method of \cite[proof of Theorem 6]{MW1973} to obtain

\begin{theorem}\label{uppertriang}
Let $N$ be the group of real strictly triangular $\ell \times \ell$ matrices,
$\gm_r$ and $\gn_r$ the algebras of Section \ref{up-triang-decomp},
and $M_r$ and $N_r$ the corresponding analytic subgroups of $N$.  Let
$\lambda = \lambda_1 + \dots + \lambda_m \in \gt^*$,
and $P(\lambda) = 
\lambda_1^{\ell - 2}\lambda_2^{\ell - 4} \dots \lambda_m^{\ell - 2m}$, 
as in {\rm (\ref{regset})}.  
Then $\pi_\lambda \in \widehat{N}$ has
distribution character
$$
\Theta_{\pi_\lambda}(f) = \tr \dot \pi_\lambda(f) =
\tfrac{1}{c}\tfrac{1}{|P(\lambda)|}\int_{\cO(\lambda)}
		\widehat{f_1}(\xi)d\nu(\xi)
$$
and $N$ has Plancherel formula
$$
f(x) = c\int_{\gt^*} \Theta_{\pi_\lambda}(r_xf)|P(\lambda)|d\lambda\,.
$$
\end{theorem}

\section{General Theory}\label{general_theory}
\setcounter{equation}{0}

\noindent
Here's what we need to extend our considerations beyond the group of upper
triangular matrices.  The connected simply connected nilpotent Lie group
should decompose as 
\begin{equation}\label{setup}
\begin{aligned}
N = &M_1M_2\dots M_{m-1}M_m \text{ where }\\
 &\text{(a) each factor $M_r$ has unitary representations with coefficients in 
$L^2(M_r/Z_r)$,} \\
 &\text{(b) each } N_r := M_1M_2\dots M_r \text{ is a normal subgroup of } N
   \text{ with } N_r = N_{r-1}\rtimes M_r \text{ semidirect,}\\
 &\text{(c) decompose }\gm_r = \gz_r + \gv_r \text{ and } \gn = \gs + \gv 
        \text{ as vector direct sums where } \\
 &\phantom{XXXX}\gs = \oplus\, \gz_r \text{ and } \gv = \oplus\, \gv_r;
    \text{ then } [\gm_r,\gz_s] = 0 \text{ and } [\gm_r,\gm_s] \subset \gv
	\text{ for } r > s\,.
\end{aligned}
\end{equation}

In order to follow the arguments leading to Theorem \ref{uppertriang},
we denote
\begin{equation}\label{c-d}
\begin{aligned}
&\text{(a) }d_r = \tfrac{1}{2}\dim(\gm_r/\gz_r) \text{ so }
	\tfrac{1}{2} \dim(\gn/\gs) = d_1 + \dots + d_m\,,
	\text{ and } c = 2^{d_1 + \dots + d_m} d_1! d_2! \dots d_m!\\
&\text{(b) }b_{\lambda_r}: (x,y) \mapsto \lambda([x,y]) 
	\text{ viewed as a bilinear form on } \gm_r/\gz_r \\
&\text{(c) }S = Z_1Z_2\dots Z_m = Z_1 \times \dots \times Z_m \text{ where } Z_r
	\text{ is the center of } M_r \\
&\text{(d) }P: \text{ polynomial } P(\lambda) = \Pf(b_{\lambda_1})
	\Pf(b_{\lambda_2})\dots \Pf(b_{\lambda_m}) \text{ on } \gs^* \\
&\text{(e) }\gt^* = \{\lambda \in \gs^* \mid P(\lambda) \ne 0\} \\
&\text{(f) } \pi_\lambda \in \widehat{N} \text{ where } \lambda \in \gt^*: 
    \text{ irreducible unitary rep. of } N = M_1M_2\dots M_m
	\text{ as in Section \ref{up-triang-planch} }
\end{aligned}
\end{equation}
Proposition \ref{supp-planch} extends immediately to this setting: Plancherel
measure is concentrated on $\{\pi_\lambda \mid \lambda \in \gt^*\}$.
It is slightly more delicate to extend Lemma \ref{split-b}, but
(\ref{setup})(c) does the job.

\begin{theorem}\label{plancherel-general}
Let $N$ be a connected simply connected nilpotent Lie group that
satisfies {\rm (\ref{setup})}.  Then Plancherel measure for $N$ is
concentrated on $\{\pi_\lambda \mid \lambda \in \gt^*\}$.
If $\lambda \in \gt^*$, and if $u$ and $v$ belong to the
representation space $\cH_{\pi_\lambda}$ of $\pi_\lambda$,  then
the coefficient $f_{u,v}(x) = \langle u, \pi_\nu(x)v\rangle$
satisfies
\begin{equation}
||f_{u,v}||^2_{L^2(N / S)} = \frac{||u||^2||v||^2}{|P(\lambda)|}\,.
\end{equation}
The distribution character $\Theta_{\pi_\lambda}$ of $\pi_{\lambda}$ satisfies
\begin{equation}
\Theta_{\pi_\lambda}(f) = c^{-1}|P(\lambda)|^{-1}\int_{\cO(\lambda)}
	\widehat{f_1}(\xi)d\nu_\lambda(\xi) \text{ for } f \in \cC(N)
\end{equation}
where $\cC(N)$ is the Schwartz space, $f_1$ is the lift 
$f_1(\xi) = f(\exp(\xi))$, $\widehat{f_1}$ is its classical Fourier transform,
$\cO(\lambda)$ is the coadjoint orbit $\Ad^*(N)\lambda = \gv^* + \lambda$,
and $d\nu_\lambda$ is the translate of normalized Lebesgue measure from
$\gv^*$ to $\Ad^*(N)\lambda$.  The Plancherel formula on $N$ is
\begin{equation}
f(x) = c\int_{\gt^*} \Theta_{\pi_\lambda}(r_xf) |P(\lambda)|d\lambda
	\text{ for } f \in \cC(N).
\end{equation}
\end{theorem}
\begin{definition}\label{stepwise2}
{\rm The representations $\pi_\lambda$ of (\ref{c-d}(f)) are the
{\it stepwise square integrable} representations of $N$ relative to
(\ref{setup}).}\hfill $\diamondsuit$
\end{definition} 

\section{Iwasawa Decompositions}
\label{iwasawa}
\setcounter{equation}{0}
Let $G$ be a real reductive Lie group.  We now carry out the program of 
Section \ref{general_theory} for the groups $N$ of Iwasawa
decompositions $G = KAN$.  Let $m = \rank_\R G = \dim_\R A$ and
notice that we've done the case $G = SL(m + 1;\R)$.
The idea is to use the Kostant cascade construction of strongly orthogonal
roots: $\beta_1$ is the maximal root, $\beta_{r+1}$ is a maximum among
the positive roots orthogonal to $\{\beta_1, \dots , \beta_r\}$, etc.  
\medskip

We fix an Iwasawa decomposition $G = KAN$.  As usual, write $\gk$ for the Lie 
algebra of $K$, $\ga$ for the Lie algebra of $A$, and $\gn$ for the
Lie algebra of $N$.  Complete $\ga$ to a Cartan subalgebra $\gh$ of $\gg$.
Then $\gh = \gt + \ga$ with $\gt = \gh \cap \gk$.  Now we have root systems
\begin{itemize}
\item $\Delta(\gg_\C,\gh_\C)$: roots of $\gg_\C$ relative to $\gh_\C$ 
(ordinary roots),

\item $\Delta(\gg,\ga)$: roots of $\gg$ relative to $\ga$ (restricted roots), 

\item $\Delta_0(\gg,\ga) = \{\alpha \in \Delta(\gg,\ga) \mid 
	2\alpha \notin \Delta(\gg,\ga)\}$ (nonmultipliable restricted roots).
\end{itemize}
\noindent Here $\Delta(\gg,\ga) = \{\gamma|_\ga \mid \gamma \in \Delta(\gg_\C,\gh_\C) 
\text{ and } \gamma|_\ga \ne 0\}$.  Further, 
$\Delta(\gg,\ga)$ and $\Delta_0(\gg,\ga)$ are root 
systems in the usual sense.  Any positive root system 
$\Delta^+(\gg_\C,\gh_\C) \subset \Delta(\gg_\C,\gh_\C)$ defines positive 
systems
\begin{itemize}
\item $\Delta^+(\gg,\ga) = \{\gamma|_\ga \mid \gamma \in 
\Delta^+(\gg_\C,\gh_\C) 
\text{ and } \gamma|_\ga \ne 0\}$ and $\Delta_0^+(\gg,\ga) =
\Delta_0(\gg,\ga) \cap \Delta^+(\gg,\ga)$.
\end{itemize}
\noindent We can (and do) choose $\Delta^+(\gg,\gh)$ so that 
\begin{itemize}
\item$\gn$ is the sum of the positive restricted root spaces and
\item if $\gamma \in \Delta(\gg_\C,\gh_\C)$ and $\gamma|_\ga \in
\Delta^+(\gg,\ga)$ then $\gamma \in \Delta^+(\gg_\C,\gh_\C)$.
\end{itemize}
\medskip

Two roots are called {\em strongly orthogonal} if their sum and their
difference are not roots.  Then they are orthogonal.  We define
\begin{equation}
\begin{aligned}
&\beta_1 \in \Delta^+(\gg,\ga) \text{ is a maximal positive restricted root
and }\\
& \beta_{r+1} \in \Delta^+(\gg,\ga) \text{ is a maximum among the roots of }
\Delta^+(\gg,\ga) \text{ orthogonal to all } \beta_i \text{ with } i \leqq r
\end{aligned}
\end{equation}
Then the $\beta_r$ are mutually strongly orthogonal.  This is Kostant's
cascade construction.  Note that each $\beta_r \in \Delta_0^+(\gg,\ga)$.
Also note that $\beta_1$ is unique if and only if $\Delta(\gg,\ga)$ is
irreducible.
\medskip

For $1\leqq r \leqq m$ define 
\begin{equation}\label{layers}
\begin{aligned}
&\Delta^+_1 = \{\alpha \in \Delta^+(\gg,\ga) \mid \beta_1 - \alpha \in \Delta^+(\gg,\ga)\} 
\text{ and }\\
&\Delta^+_{r+1} = \{\alpha \in \Delta^+(\gg,\ga) \setminus (\Delta^+_1 \cup \dots \cup \Delta^+_r)
	\mid \beta_{r+1} - \alpha \in \Delta^+(\gg,\ga)\}.
\end{aligned}
\end{equation} 
\begin{lemma} \label{fill-out}
If $\alpha \in \Delta^+(\gg,\ga)$ then either 
$\alpha \in \{\beta_1, \dots , \beta_m\}$
or $\alpha$ belongs to exactly one of the sets $\Delta^+_r$\,.
\end{lemma}

\begin{proof}  Suppose that $\alpha \not\in \{\beta_1, \dots , \beta_m\}$
and that $\alpha \not\in \Delta^+_r$ for any $r$.  As $\alpha \notin \Delta^+_1$ it is
strongly orthogonal to $\beta_1$.  Then as $\alpha \notin \Delta^+_2$ it is
strongly orthogonal to $\beta_2$ as well.  Continuing, $\alpha$ is strongly
orthogonal to each of the $\beta_r$, contradicting maximality of
$\{\beta_1, \dots , \beta_m\}$.
\end{proof}

\begin{lemma}\label{layers2}
The set $\Delta^+_r\cup \{\beta_r\}  
= \{\alpha \in \Delta^+ \mid \alpha \perp \beta_i \text{ for } i < r
\text{ and } \langle \alpha, \beta_r\rangle > 0\}.$
In particular, $[\gm_r,\gm_s] \subset \gm_t$ where $t = \min\{r,s\}$.
\end{lemma}
\begin{proof}
Let $\alpha \in \Delta^+$ such that (i) $\alpha \perp \beta_i \text{ for } 
i < r$ and (ii) $\langle \alpha, \beta_r\rangle > 0$.  Here (ii) shows that
$\beta_r - \alpha$ is a root.  If it's negative then $\alpha > \beta_r$,
contradicting maximality of $\beta_r$ for the property of being orthogonal
to $\beta_i \text{ for every } i < r$. So $\beta_r - \alpha \in \Delta^+$.
Let $s$ be the smallest integer such that $\alpha \in \Delta^+_s$.  This
relies on Lemma \ref{fill-out}.  As argued a moment ago, 
$\beta_s + \alpha$ is not a root.  If $s < r$ then (i) says that $\alpha$
is strongly orthogonal to $\beta_s$\,, contradicting $\alpha \in \Delta^+_s$.
Thus $r = s$ and $\alpha \in \Delta^+_r$\,.
\medskip

Conversely we want to show that $\alpha \in \Delta^+_r$ implies
$\alpha \perp \beta_i \text{ for } i < r
\text{ and } \langle \alpha, \beta_r\rangle > 0$.  This is clear for 
$r = 1$.  We assume it for $r < t$, for a fixed $t \leqq m$, and prove it 
for $r = t$.
Let $\alpha \in \Delta^+_t$.  If $\alpha \not\perp \beta_r$ where $r < t$,
and $\alpha + \beta_r$ is a root, then $\alpha + \beta_r \in \Delta^+_s$
where $s < r$, and $\langle \alpha + \beta_r, \beta_s \rangle > 0$.
That is impossible because $\alpha \perp \beta_s \perp \beta_r$.
If $\alpha \not\perp \beta_r$ now $\beta_r - \alpha$ is a root.  It is
positive by the maximality property of $\beta_r$, so $\alpha \in
\Delta^+_r$, contradicting $\alpha \in \Delta^+_t$ with $r < t$.  
Thus $\alpha \perp \beta_r$ for all $r < t$.  As argued before,
$\alpha + \beta_t$ is not a root.  Since $\beta_t - \alpha \in \Delta^+$
now $\langle\alpha,\beta_t\rangle > 0$.  That completes the induction.
\medskip

Finally, let $\alpha \in \Delta^+_r \cup \{\beta_r\}$, 
$\gamma \in \Delta^+_s \cup \{\beta_s\}$, and
$t = \min\{r,s\}$.  Suppose that $\alpha + \gamma$ is a root.  If
$i < t$ then $\langle \alpha + \gamma, \beta_i\rangle =
\langle \alpha, \beta_i\rangle + \langle \gamma, \beta_i\rangle
= 0$, and $\langle \alpha + \gamma, \beta_t\rangle >0$ because at least
one of $\langle \alpha, \beta_t\rangle$ and $\langle\gamma, \beta_t\rangle$
is positive.
\end{proof}

Lemma \ref{fill-out} shows that the Lie algebra $\gn$ of $N$ is the
vector space direct sum of its subspaces
\begin{equation}\label{def-m}
\gm_r = \gg_{\beta_r} + {\sum}_{\Delta^+_r}\, \gg_\alpha 
\text{ for } 1\leqq r\leqq m
\end{equation}
and Lemma \ref{layers2} shows that $\gn$ has an increasing foliation
by ideals
\begin{equation}\label{def-filtration}
\gn_r = \gm_1 + \gm_2 + \dots + \gm_r \text{ for } 1 \leqq r \leqq m.
\end{equation}
Now we will see that the corresponding group level decomposition
$N = M_1M_2\dots M_m$ and the semidirect product decompositions
$N_r = N_{r-1}\rtimes M_r$ satisfy all the requirements of (\ref{setup}).
\medskip

The structure of $\Delta^+_r$, and later of $\gm_r$, is exhibited by a 
particular Weyl group element of $\Delta(\gg,\ga)$.  Denote
\begin{equation}
s_{\beta_r} \text{ is the Weyl group reflection in } \beta_r
\text{ and } \sigma_r: \Delta(\gg,\ga) \to \Delta(\gg,\ga) \text{ by }
\sigma_r(\alpha) = -s_{\beta_r}(\alpha).
\end{equation}
Note that $\sigma_r(\beta_s) = -\beta_s$ for $s \ne r$, $+\beta_s$ if $s = r$.
If $\alpha \in \Delta^+_r$ we still have $\sigma_r(\alpha) \perp \beta_i$
for $i < r$ and $\langle \sigma_r(\alpha), \beta_r\rangle > 0$.  If
$\sigma_r(\alpha)$ is negative then $\beta_r - \sigma_r(\alpha) > \beta_r$
contradicting the maximality property of $\beta_r$.  Thus, using 
Lemma \ref{layers2}, $\sigma_r(\Delta^+_r) = \Delta^+_r$.

\begin{lemma} \label{layers-nilpotent}
If $\alpha \in \Delta^+_r$ then $\alpha + \sigma_r(\alpha) = \beta_r$.
{\rm (}Of course it is possible that 
$\alpha = \sigma_r(\alpha) = \tfrac{1}{2}\beta_r$ when 
$\tfrac{1}{2}\beta_r$ is a root.{\rm ).}
If $\alpha, \alpha' \in \Delta^+_r$ and $\alpha + \alpha' \in \Delta(\gg,\ga)$
then $\alpha + \alpha' = \beta_r$\,.
\end{lemma}

\begin{proof} 
If $\alpha \in \Delta^+_r$ with $\sigma_r(\alpha) = \alpha$ then 
$s_{\beta_r}(\alpha) = -\alpha$ so $\alpha$ is proportional to $\beta_r$.
As $\beta_r$ is nonmultipliable and $\langle \alpha , \beta_r \rangle > 0$
that forces $\alpha = \tfrac{1}{2}\beta_r$.  In particular $\alpha +
\sigma_r(\alpha) = \beta_r$.  
\medskip

Now suppose $\alpha \in \Delta^+_r$ with $\sigma_r(\alpha) \ne \alpha$.  Then
$\alpha + \sigma_r(\alpha) = 
\alpha - s_{\beta_r}(\alpha) = \alpha -(\alpha - \tfrac{2\langle \alpha , \beta_r \rangle}
{\langle \beta_r,\beta_r\rangle}\beta_r) =
\tfrac{2\langle \alpha , \beta_r \rangle}
{\langle \beta_r,\beta_r\rangle}\beta_r$.  
As $\langle \alpha , \beta_r \rangle > 0$ 
and $\beta_r$ is nonmultipliable this forces $\alpha + \sigma_r(\alpha)
= \beta_r$.
\medskip

Suppose that there exist $\alpha, \alpha' \in \Delta^+_r$
such that $\alpha + \alpha' = \alpha'' \in \Delta(\gg,\ga)$ but 
$\alpha'' \ne \beta_r$\,.  Fix such a pair $\{\alpha,\alpha'\}$
with $\alpha$ maximal for that property.  Then $\alpha''$ lacks that
property.  So $\beta_r = \alpha'' + \sigma_r(\alpha'') = 
(\alpha + \alpha') + \sigma_r(\alpha + \alpha')
= (\alpha + \sigma_r(\alpha)) + (\alpha' + \sigma_r(\alpha')) = 2\beta_r$\,.
Thus the specified $\alpha$ cannot exist.
\end{proof}

Now we are in a position to start the proof of the main technical result
of this section -- that the $M_r$ have square integrable representations.
For that it suffices to consider the case where $\gg$ is simple as a real
Lie algebra and run through some possibilities:

\begin{lemma}\label{pairing}
Let $\gn$ be a nilpotent Lie algebra, $\gz$ its center, and $\gv$ a
vector space complement to $\gz$ in $\gn$.  Suppose that we have
vector space direct sum decompositions
$\gv = \gu + \gu'$, $\gu = \sum \gu_a$ and $\gu' = \sum \gu'_a$\,, and
$\gz = \sum \gz_b$ with $\dim \gz_b = 1$.  Suppose further that 
{\rm (i)} each $[\gu_a, \gu_a] = 0 = [\gu'_a,\gu'_a]$, 
{\rm (ii)} if $a_1 \ne a_2$ then $[\gu_{a_1} , \gu'_{a_2}] = 0$ and 
{\rm (iii)} for each index $a$ there is a unique $b_a$ such that 
$\gu_a \otimes \gu'_a \to \gz_{b_a}$\,, by $u \otimes u' \mapsto [u,u']$,
is a nondegenerate pairing.  Then $\gn$ is a direct sum of Heisenberg
algebras  $\gz_{b_a} + \gu_a + \gu'_a$ and the commutative algebra that is the
sum of the remaining $\gz_b$\,.
\end{lemma}

\begin{lemma}\label{split-case}
If $\gg$ is the split real form of $\gg_\C$ then each $M_r$ has square
integrable representations.
\end{lemma}

\begin{lemma}\label{complex-case}
If $\gg$ is simple but not absolutely simple then each $M_r$ has square
integrable representations.
\end{lemma}

\begin{lemma}\label{quaternion-linear}
If $G$ is the quaternion special linear group $SL(n;\H)$ then 
$M_1$ has square integrable representations.
\end{lemma}

\begin{lemma}\label{cayley-collineations}
If $G$ is the group $E_{6,F_4}$ of collineations of the Cayley
projective plane then $M_1$ has square integrable representations.
\end{lemma}

\begin{lemma}\label{start-induction}
The group $M_1$ has square integrable representations.
\end{lemma}

\begin{lemma}\label{abs-simple-case}
If $\gg$ is absolutely simple then each $M_r$ has square
integrable representations.
\end{lemma}

\begin{proof}({\rm Lemma \ref{pairing}}.)  The assertion is obvious.
\end{proof}

\begin{proof}({\rm Lemma \ref{split-case}}.)
This is the case where $\ga$ is a Cartan subalgebra of $\gg$ and
$\Delta(\gg_\C,\gh_\C) = \Delta(\gg_\C,\ga_\C)$ consists of the
$\C$--linear extensions of the roots in $\Delta(\gg,\ga)$.  All
roots are indivisible, so Lemma \ref{layers-nilpotent} divides $\Delta^+_r$
into two disjoint subsets, thus divides $\sum_{\alpha \in \Delta^+_r}\gg_\alpha$
as a direct sum $\gu \oplus \gu'$ of two subspaces, such that 
those subspaces satisfy the
conditions of Lemma \ref{pairing}.  As $\gm_r$ has $1$--dimensional center
$\gg_{\beta_r}$ it follows that $\gm_r$ is a Heisenberg algebra.  Now
$M_r$ is a Heisenberg group and thus has square integrable representations.
\end{proof}

\begin{proof}({\rm Lemma \ref{complex-case}}.)
This is the case where $\gg$ is the underlying real structure of a
complex simple Lie algebra $\gs$.  The Cartan subalgebra $\gh = \gt + \ga$
of $\gs$ is given by $\gt = \sqrt{-1}\ga$, and $\ga$ is the (real)
subspace on which the roots take real values.  As a real Lie algebra, 
$\gn \cong \sum_{\alpha \in \Delta^+(\gs,\ga)} \gs_\alpha$.
\medskip

Let $\gg'$ denote the split real form of $\gs$. In the Iwasawa decomposition
$G' = K'A'N'$ now $\ga' = \ga$, $\gn'$ is a real form of $\gn$, and for
each $r$ the algebra 
$\gm'_r := \gm_r \cap \gn'$ is a real form of $\gm_r$.  From the
latter, \cite[Theorem 2.1]{W1979} says that the corresponding group $M'_r$ has
square integrable representations if and only if its complexification $M_r$
has square integrable representations.  However, Lemma \ref{split-case}
says that $M'_r$ has square integrable representations.  Our assertion follows.
\end{proof}

\begin{proof}({\rm Lemma \ref{quaternion-linear}}.)
This is the case where 
$\gg = \gs\gl(n;\H)$.  In the usual root ordering, $\widetilde{\beta}_1
= \psi_1 + \dots + \psi_{2n-1}$ and $\psi_i|_\ga = 0$ just when $i$ is odd.
Thus the ordinary roots that restrict to $\beta_1$ are
$\psi_1 + \dots + \psi_{2n-1}$, $\psi_2 + \dots + \psi_{2n-1}$,
$\psi_1 + \dots + \psi_{2n-2}$ and $\psi_2 + \dots + \psi_{2n-2}$;
their root spaces span the center $\gz_1$ of $\gm_1$.
Further $\Delta^+_1$ consists of the restrictions of pairs of roots that
\begin{itemize}

\item sum to $\psi_1 + \dots + \psi_{2n-1}$:
$\{\psi_1 + \dots + \psi_j\, , \psi_{j+1} + \dots + \psi_{2n-1}\}$,
$1 \leqq j < 2n-1$,

\item sum to $\psi_2 + \dots + \psi_{2n-1}$:
$\{\psi_2 + \dots + \psi_j\, , \psi_{j+1} + \dots + \psi_{2n-1}\}$,
$2 \leqq j < 2n-1$,

\item sum to $\psi_1 + \dots + \psi_{2n-2}$:
$\{\psi_1 + \dots + \psi_j\, , \psi_{j+1} + \dots + \psi_{2n-2}\}$,
$1 \leqq j < 2n-2$,

\item sum to $\psi_2 + \dots + \psi_{2n-2}$:
$\{\psi_2 + \dots + \psi_j\, , \psi_{j+1} + \dots + \psi_{2n-2}\}$.
$2 \leqq j < 2n-2$.
\end{itemize}
Their root spaces span a complement $\gv_1$ to $\gz_1$ in $\gm_1$.
Eliminating duplicates, the set of ordinary roots that restrict to elements of
$\Delta^+_1$ is $\{\psi_1 + \dots + \psi_j\,;
\psi_{j+1} + \dots + \psi_{2n-1}\,; \psi_2 + \dots + \psi_j\,;
\psi_{j+1} + \dots + \psi_{2n-2}\}$.  Now let $\lambda \in \gz_1^*$
be zero on the root spaces for $\psi_2 + \dots + \psi_{2n-1}$ and
$\psi_1 + \dots + \psi_{2n-2}$, nonzero on the root spaces for
$\psi_1 + \dots + \psi_{2n-1}$ and $\psi_2 + \dots + \psi_{2n-2}$.
Then the corresponding antisymmetric bilinear form $b_\lambda$ on
$\gv_1$ is nonsingular.  Thus $M_1$ has square integrable (modulo its
center) representations.
\end{proof}

\begin{proof}({\rm Lemma \ref{cayley-collineations}}.)
This is the case where
$\gg = \ge_{6,F_4}$.  Then $\rank_\R\gg = 2$.  In the Bourbaki 
order for the simple roots \setlength{\unitlength}{.35 mm}
\begin{picture}(75,23)
\put(10,12){\circle{2}}
\put(8,15){$\psi_1$}
\put(11,12){\line(1,0){13}}
\put(25,12){\circle{2}}
\put(23,15){$\psi_3$}
\put(26,12){\line(1,0){13}}
\put(40,12){\circle{2}}
\put(38,15){$\psi_4$}
\put(41,12){\line(1,0){13}}
\put(55,12){\circle{2}}
\put(53,15){$\psi_5$}
\put(56,12){\line(1,0){13}}
\put(70,12){\circle{2}}
\put(68,15){$\psi_6$}
\put(40,11){\line(0,-1){13}}
\put(40,-3){\circle{2}}
\put(43,-3){$\psi_2$}
\end{picture},
$\widetilde{\beta}_1 
= \psi_1 + 2\psi_2 + 2\psi_3 + 3\psi_4 + 2\psi_5 + \psi_6$, and the roots
that restrict to $0$ on $\ga$ are $\psi_2, \psi_3, \psi_4$ and $\psi_5$.
So $\Delta^+_1$ consists of the restrictions of pairs of roots that
\begin{itemize}
\item sum to $\psi_1 + 2\psi_2 + 2\psi_3 + 3\psi_4 + 2\psi_5 + \psi_6$:
$$
\begin{aligned}
\{ &\{\psi_2\,,\, \psi_1 + \psi_2 + 2\psi_3 + 3\psi_4 + 2\psi_5 + \psi_6\},\\
         &\{\psi_2 + \psi_4\,,\, \psi_1 + \psi_2 + 2\psi_3 + 2\psi_4 + 2\psi_5
                + \psi_6\},\\
         &\{\psi_2 + \psi_3 + \psi_4\,,\, \psi_1 + \psi_2 + \psi_3 + 2\psi_4 +
                2\psi_5 + \psi_6\},\\
         &\{\psi_2 + \psi_4 + \psi_5\,,\, \psi_1 + \psi_2 + 2\psi_3 + 2\psi_4 +
                \psi_5 + \psi_6\},\\
         &\{\psi_1 + \psi_2 + \psi_3 + \psi_4\,,\, \psi_2 + \psi_3 + 2\psi_4 +
                2\psi_5 + \psi_6\},\\
         &\{\psi_2 + \psi_3 + \psi_4 + \psi_5\,,\, \psi_1 + \psi_2 + \psi_3 +
                2\psi_4 + \psi_5 + \psi_6\},\\
         &\{\psi_2 + \psi_4 + \psi_5 + \psi_6\,,\, \psi_1 + \psi_2 + 2\psi_3 +
                2\psi_4 + \psi_5\},\\
         &\{\psi_1 + \psi_2 + \psi_3 + \psi_4 + \psi_5\,,\, \psi_2 + \psi_3 +
                2\psi_4 + \psi_5 + \psi_6\},\\
         &\{\psi_2 + \psi_3 + 2\psi_4 + \psi_5\,,\, \psi_1 + \psi_2 + \psi_3 +
                \psi_4 + \psi_5 + \psi_6\},\\
         &\{\psi_2 + \psi_3 + \psi_4 + \psi_5 + \psi_6\,,\,  \psi_1 + \psi_2 +
                \psi_3 + 2\psi_4 + \psi_5\} \}
\end{aligned}
$$
\item sum to $\psi_1 +  \psi_2 + 2\psi_3 + 3\psi_4 + 2\psi_5 + \psi_6$:
$$
\begin{aligned}
\{ &\{\psi_4\,,\, \psi_1 + \psi_2 + 2\psi_3 + 2\psi_4 + 2\psi_5
                + \psi_6\},\\ 
         &\{\psi_3 + \psi_4\,,\, \psi_1 + \psi_2 + \psi_3 + 2\psi_4 +
                2\psi_5 + \psi_6\},\\ 
         &\{\psi_4 + \psi_5\,,\, \psi_1 + \psi_2 + 2\psi_3 + 2\psi_4 +
                \psi_5 + \psi_6\},\\ 
         &\{\psi_1 + \psi_3 + \psi_4\,,\, \psi_2 + \psi_3 + 2\psi_4 +
                2\psi_5 + \psi_6\},\\
         &\{\psi_3 + \psi_4 + \psi_5\,,\, \psi_1 + \psi_2 + \psi_3 +
                2\psi_4 + \psi_5 + \psi_6\},\\ 
         &\{\psi_4 + \psi_5 + \psi_6\,,\, \psi_1 + \psi_2 + 2\psi_3 +
                2\psi_4 + \psi_5\},\\
         &\{\psi_1 + \psi_3 + \psi_4 + \psi_5\,,\, \psi_2 + \psi_3 +
                2\psi_4 + \psi_5 + \psi_6\},\\
         &\{\psi_2 + \psi_3 + 2\psi_4 + \psi_5\,,\, \psi_1 + \psi_3 +
                \psi_4 + \psi_5 + \psi_6\},\\ 
         &\{\psi_3 + \psi_4 + \psi_5 + \psi_6\,,\,  \psi_1 + \psi_2 +
                \psi_3 + 2\psi_4 + \psi_5\} \}
\end{aligned}
$$
\item sum to $\psi_1 +  \psi_2 + 2\psi_3 + 2\psi_4 + 2\psi_5 + \psi_6$:
$$
\begin{aligned}
\{ &\{\psi_3\,,\, \psi_1 + \psi_2 + \psi_3 + 2\psi_4 +
                2\psi_5 + \psi_6\},\\
         &\{\psi_5\,,\, \psi_1 + \psi_2 + 2\psi_3 + 2\psi_4 +
                \psi_5 + \psi_6\},\\
         &\{\psi_1 + \psi_3\,,\, \psi_2 + \psi_3 + 2\psi_4 +
                2\psi_5 + \psi_6\},\\
         &\{\psi_3 + \psi_4 + \psi_5\,,\, \psi_1 + \psi_2 + \psi_3 +
                \psi_4 + \psi_5 + \psi_6\},\\
         &\{\psi_5 + \psi_6\,,\, \psi_1 + \psi_2 + 2\psi_3 +
                2\psi_4 + \psi_5\},\\
         &\{\psi_1 + \psi_3 + \psi_4 + \psi_5\,,\, \psi_2 + \psi_3 +
                \psi_4 + \psi_5 + \psi_6\},\\
         &\{\psi_2 + \psi_3 + \psi_4 + \psi_5\,,\, \psi_1 + \psi_3 +
                \psi_4 + \psi_5 + \psi_6\},\\
         &\{\psi_3 + \psi_4 + \psi_5 + \psi_6\,,\,  \psi_1 + \psi_2 +
                \psi_3 + \psi_4 + \psi_5\} \}
\end{aligned}
$$
\item sum to $\psi_1 +  \psi_2 +  \psi_3 + 2\psi_4 + 2\psi_5 + \psi_6$:
$$
\begin{aligned}
\{ &\{\psi_5\,,\, \psi_1 + \psi_2 + \psi_3 + 2\psi_4 +
                \psi_5 + \psi_6\},\\
         &\{\psi_1\,,\, \psi_2 + \psi_3 + 2\psi_4 +
                2\psi_5 + \psi_6\},\\
         &\{\psi_4 + \psi_5\,,\, \psi_1 + \psi_2 + \psi_3 +
                \psi_4 + \psi_5 + \psi_6\},\\
         &\{\psi_5 + \psi_6\,,\, \psi_1 + \psi_2 + \psi_3 +
                2\psi_4 + \psi_5\},\\
         &\{\psi_1 + \psi_3 + \psi_4 + \psi_5\,,\, \psi_2 +
                \psi_4 + \psi_5 + \psi_6\},\\
         &\{\psi_2 + \psi_4 + \psi_5\,,\, \psi_1 + \psi_3 +
                \psi_4 + \psi_5 + \psi_6\},\\
         &\{\psi_4 + \psi_5 + \psi_6\,,\,  \psi_1 + \psi_2 +
                \psi_3 + \psi_4 + \psi_5\} \}
\end{aligned}
$$
\item sum to $\psi_1 +  \psi_2 + 2\psi_3 + 2\psi_4 +  \psi_5 + \psi_6$:
$$
\begin{aligned}
\{ &\{\psi_3\,,\, \psi_1 + \psi_2 + \psi_3 + 2\psi_4 +
                \psi_5 + \psi_6\},\\
         &\{\psi_1 + \psi_3\,,\, \psi_2 + \psi_3 + 2\psi_4 +
                \psi_5 + \psi_6\},\\
         &\{\psi_3 + \psi_4\,,\, \psi_1 + \psi_2 + \psi_3 +
                \psi_4 + \psi_5 + \psi_6\},\\
         &\{\psi_6\,,\, \psi_1 + \psi_2 + 2\psi_3 +
                2\psi_4 + \psi_5\},\\
         &\{\psi_1 + \psi_3 + \psi_4\,,\, \psi_2 + \psi_3 +
                \psi_4 + \psi_5 + \psi_6\},\\
         &\{\psi_2 + \psi_3 + \psi_4\,,\, \psi_1 + \psi_3 +
                \psi_4 + \psi_5 + \psi_6\},\\
         &\{\psi_3 + \psi_4 + \psi_5 + \psi_6\,,\,  \psi_1 + \psi_2 +
                \psi_3 + \psi_4\} \}
\end{aligned}
$$

\item sum to $\psi_1 +  \psi_2 +  \psi_3 + 2\psi_4 +  \psi_5 + \psi_6$:
$$
\begin{aligned}
\{ &\{\psi_1\,,\, \psi_2 + \psi_3 + 2\psi_4 +
                \psi_5 + \psi_6\},\\
         &\{\psi_4\,,\, \psi_1 + \psi_2 + \psi_3 +
                \psi_4 + \psi_5 + \psi_6\},\\
         &\{\psi_6\,,\, \psi_1 + \psi_2 + \psi_3 +
                2\psi_4 + \psi_5\},\\
         &\{\psi_1 + \psi_3 + \psi_4\,,\, \psi_2 +
                \psi_4 + \psi_5 + \psi_6\},\\
         &\{\psi_2 + \psi_4\,,\, \psi_1 + \psi_3 +
                \psi_4 + \psi_5 + \psi_6\},\\
         &\{\psi_4 + \psi_5 + \psi_6\,,\,  \psi_1 + \psi_2 +
                \psi_3 + \psi_4\} \}
\end{aligned}
$$
\item sum to $\psi_1 +  \psi_2 +  \psi_3 + \psi_4 +  \psi_5 + \psi_6$:
$$
\begin{aligned}
\{ &\{\psi_1\,,\, \psi_2 + \psi_3 + \psi_4 +
                \psi_5 + \psi_6\},\\
         &\{\psi_6\,,\, \psi_1 + \psi_2 + \psi_3 +
                \psi_4 + \psi_5\},\\
         &\{\psi_1 + \psi_3\,,\, \psi_2 +
                \psi_4 + \psi_5 + \psi_6\},\\
         &\{\psi_2\,,\, \psi_1 + \psi_3 +
                \psi_4 + \psi_5 + \psi_6\},\\
         &\{\psi_5 + \psi_6\,,\,  \psi_1 + \psi_2 +
                \psi_3 + \psi_4\} \}
\end{aligned}
$$
\item sum to $\psi_1 + \psi_3 + \psi_4 +  \psi_5 + \psi_6$:
$$
\begin{aligned}
\{ &\{\psi_1\,,\, \psi_3 + \psi_4 + \psi_5 + \psi_6\},\\
         &\{\psi_1 + \psi_3\,,\, \psi_4 + \psi_5 + \psi_6\},\\
	 &\{\psi_1 + \psi_3 + \psi_4\,,\, \psi_5 + \psi_6\},\\
	 &\{\psi_1 + \psi_3 + \psi_4 + \psi_5\,,\, \psi_6\}\}
\end{aligned}
$$
\end{itemize}
Eliminating duplicates, the set of ordinary roots that restrict to
elements of $\Delta^+_1$ consists of
\begin{itemize}
\item  The  $20$ positive roots listed above in the first 
group, summing to $\psi_1 + 2\psi_2 + 2\psi_3 + 3\psi_4 + 2\psi_5 + \psi_6$.
These are the roots $\sum a_i\psi_i$ for which $a_2 = 1$ and 
$(a_1,a_6)$ is either $(1,0)$ or $(0,1)$.  We denote the sum of their
root spaces by $\gv_1'$.

\item  The $8$ positive roots listed above in the last group, summing to 
$\psi_1 + \psi_3 + \psi_4 +  \psi_5 + \psi_6$.  These are the 
roots $\sum a_i\psi_i$ for which $a_2 = 0$ and
$(a_1,a_6)$ is either $(1,0)$ or $(0,1)$.  We denote the sum of their
root spaces by $\gv_1''$.
\end{itemize}
Now the space $\gv_1 := \gv_1' + \gv_1''$ is a vector space complement
to $\gz_1$ in $\gm_1$.  Let $\lambda \in \gz_1^*$ be zero on 
the root spaces for 
$\psi_1 +  \psi_2 + 2\psi_3 + 3\psi_4 + 2\psi_5 + \psi_6$,
$\psi_1 +  \psi_2 + 2\psi_3 + 2\psi_4 + 2\psi_5 + \psi_6$,
$\psi_1 +  \psi_2 +  \psi_3 + 2\psi_4 + 2\psi_5 + \psi_6$,
$\psi_1 +  \psi_2 + 2\psi_3 + 2\psi_4 +  \psi_5 + \psi_6$,
$\psi_1 +  \psi_2 +  \psi_3 + 2\psi_4 +  \psi_5 + \psi_6$ and
$\psi_1 +  \psi_2 +  \psi_3 + \psi_4 +  \psi_5 + \psi_6$,
and nonzero on the root spaces for 
$\psi_1 + 2\psi_2 + 2\psi_3 + 3\psi_4 + 2\psi_5 + \psi_6$
and $\psi_1 + \psi_3 + \psi_4 +  \psi_5 + \psi_6$.  Then the corresponding
antisymmetric bilinear form $b_\lambda$ on $\gv_1$ is nonsingular, so
$M_1$ has square integrable (modulo its center) representations.
\end{proof}

\begin{proof}({\rm Lemma \ref{start-induction}}.)
It suffices to consider the case where $\gg_\C$ is a simple complex 
Lie algebra, but $\gg$
need not be its split real form.  We do, however, assume that it is not
the compact real form, for in that case $N = \{1\}$.
\medskip

Suppose first that $\dim \gg_{\beta_1} = 1$.  In other words the highest
root in $\Delta^+(\gg_\C,\gh_\C)$, call it $\widetilde{\beta}_1$, is the
only root of $\Delta(\gg_\C,\gh_\C)$ that restricts to $\beta_1$.  
Applying Lemma \ref{layers-nilpotent} as in the proof of Lemma
\ref{split-case} it follows that $M_1$ has square integrable representations.
\medskip

Now suppose that $\dim \gg_{\beta_1} > 1$.  Note that the roots in
$\Delta(\gg_\C,\gh_\C)$ that restrict to $\beta_1$ are just the roots of 
the form $\widetilde{\beta}_1 - \sum t_i\gamma_i$ where every one of the
$\gamma_i|_\ga = 0$.  In particular the root(s) of the extended Dynkin diagram
of $\Delta^+(\gg_\C,\gh_\C)$, to which $-\widetilde{\beta}_1$ attaches, have
restriction $0$ on $\ga$.  We have already dealt with the cases 
$\gg = \gs\gl(n;\H)$ and $\gg = \ge_{6,F_4}$, so there remain only a few 
easy cases:
\smallskip 

Case $\gg = \gs\go(1,n)$.  Then $\rank_\R\gg = 1$,
$\Delta^+(\gg,\ga) = \{\widetilde{\beta}_1\}$ and $M_1 = N$ is abelian.
In particular $M_1$ has square integrable (modulo its center) representations.
\smallskip

Case $\gg = \gs\gu(1,n)$.  Then $\rank_\R\gg = 1$,
$\Delta^+(\gg,\ga) = \{\widetilde{\beta}_1, \tfrac{1}{2}\widetilde{\beta}_1\}$,
and $M_1 = N$ is a Heisenberg group $\Im\C + \C^{n-1}$.  
In particular $M_1$ has square integrable representations.
\smallskip

Case $\gg = \gs\gp(p,q)$, $p \leqq q$.  Then $\rank_\R\gg = p$,
$\Delta^+(\gg,\ga) = \{\widetilde{\beta}_1, \tfrac{1}{2}\widetilde{\beta}_1\}$,
and $M_1 = N$ is a quaternionic Heisenberg group $\Im \H + \H^s$.
In particular $M_1$ has square integrable (modulo its center) representations.

Case $\gg = \gf_{4,B_4}$.  Then $\rank_\R\gg = 1$,
$\Delta^+(\gg,\ga) = \{\widetilde{\beta}_1, \tfrac{1}{2}\widetilde{\beta}_1\}$,
and $M_1 = N$ is an octonionic Heisenberg group $\Im\O + \O$.
In particular $M_1$ has square integrable (modulo its center) representations.
\end{proof}

\begin{proof}({\rm Lemma \ref{abs-simple-case}}.)
This is the case where $\gg_\C$ is a simple complex Lie algebra, but $\gg$
need not be its split real form.  We do, however, assume that it is not
the compact real form, for in that case $N = \{1\}$. Then $\beta_1(\gt) = 0$. 
Note that $\beta_1$ is the restriction to $\ga$ of the highest root 
$\widetilde{\beta}_1$ in $\Delta^+(\gg_\C,\gh_\C)$ and that 
$\widetilde{\beta}_1$ is a long root.  Thus $\widetilde{\beta}_1$ is the
only root in $\Delta^+(\gg_\C,\gh_\C)$ that restricts of $\beta_1$.
Applying Lemma \ref{layers-nilpotent} as in the proof of Lemma
\ref{split-case} it follows that $M_1$ has square integrable representations.
This starts the induction. 
\medskip

Suppose we know that $M_1$, \dots , $M_r$ have square integrable 
representations and that $r < m$.  Let $\gg_r$ be the semisimple
subalgebra of $\gg$ generated by $\gh$ and all (restricted) root spaces for
simple roots that are orthogonal to $\beta_1$, $\beta_2$, \dots , $\beta_r$.
Then $\beta_{r+1}$ is a maximum among the positive restricted roots of
$\gg_r$ and $\gm_{r+1}$ is the subalgebra of $\gg_r$ that is the
counterpart of $\gm_1$ for $\gg$.  Thus by the argument just above for $M_1$,
and by Lemmas \ref{split-case} and \ref{complex-case} as needed for the
simple summands of $\gg_r$, we conclude that $M_{r+1}$ has square
integrable representations.
\end{proof}

We now apply Lemmas \ref{complex-case} and \ref{abs-simple-case} to the
list (\ref{setup}) of conditions for setting up the character formula
and Plancherel formula as in Theorem \ref{plancherel-general}.  Those
Lemmas supply the key condition, that each $M_r$ has unitary representations 
with coefficients in $L^2(M_r/Z_r)$.  Lemma \ref{layers2} ensures that
each $N_r := M_1M_2\dots M_r$ is a normal subgroup of $N$ with
$N_r = N_{r-1}\rtimes M_r$ semidirect product, and then Lemma \ref{layers}
says that $N = M_1M_2\dots M_{m-1}M_m$ as needed.  The decompositions
$\gm_r = \gz_r + \gv_r$ and $\gn = \gs + \gv$ now are immediate from the
construction of the $\gm_r$.  It remains only to verify that 
$[\gm_r,\gz_s] = 0$ and $[\gm_r,\gm_s] \subset \gv$ for $r > s$.  
\medskip

Let $\alpha \in \Delta^+_r$ and $s < r$.  Then $\alpha \notin (\Delta^+_1
\cup \dots \cup \Delta^+_s)$ and $\alpha \perp \beta_i$ for $i \leqq s$.
Now $\beta_s + \alpha \in \Delta^+$ would imply $\beta_s - \alpha \in \Delta^+$,
contradicting $\alpha \notin \Delta^+_s$.  It follows that $[\gm_r,\gz_s] = 0$.
\medskip

Let $\alpha \in \Delta^+_r$ and $\alpha' \in \Delta^+_s$\,, $s < r$, with
$\alpha + \alpha' = \beta_t$.  Lemma \ref{layers2} says $s = t$ so
$\beta_s - \alpha' = \alpha$.  But then $\beta_s - \alpha' \in \Delta^+_s$
contradicting $\alpha \notin (\Delta^+_1 \cup \dots \cup \Delta^+_s)$.
We conclude that $[\gm_r,\gm_s] \subset \gv$ for $r > s$.
\medskip

Summarizing, we have just shown that Theorem \ref{plancherel-general}
applies to milradicals of minimal parabolic subgroups.  In other words,

\begin{theorem}\label{iwasawa-layers}
Let $G$ be a real reductive Lie group, $G = KAN$ an Iwasawa
decomposition, $\gm_r$ and $\gn_r$ the subalgebras of $\gn$ defined in 
{\rm (\ref{def-m})} and {\rm (\ref{def-filtration})},
and $M_r$ and $N_r$ the corresponding analytic subgroups of $N$.  
Then the $M_r$ and $N_r$ satisfy {\rm (\ref{setup})}.  In particular,
Plancherel measure for $N$ is
concentrated on $\{\pi_\lambda \mid \lambda \in \gt^*\}$.
If $\lambda \in \gt^*$, and if $u$ and $v$ belong to the
representation space $\cH_{\pi_\lambda}$ of $\pi_\lambda$,  then
the coefficient $f_{u,v}(x) = \langle u, \pi_\lambda(x)v\rangle$
satisfies
\begin{equation}
||f_{u,v}||^2_{L^2(N / S)} = \frac{||u||^2||v||^2}{|P(\lambda)|}\,.
\end{equation}
The distribution character $\Theta_{\pi_\lambda}$ of $\pi_{\lambda}$ satisfies
\begin{equation}
\Theta_{\pi_\lambda}(f) = c^{-1}|P(\lambda)|^{-1}\int_{\cO(\lambda)}
        \widehat{f_1}(\xi)d\nu_\lambda(\xi) \text{ for } f \in \cC(N)
\end{equation}
where $\cC(N)$ is the Schwartz space, $f_1$ is the lift
$f_1(\xi) = f(\exp(\xi))$, $\widehat{f_1}$ is its classical Fourier transform,
$\cO(\lambda)$ is the coadjoint orbit $\Ad^*(N)\lambda = \gv^* + \lambda$,
and $d\nu_\lambda$ is the translate of normalized Lebesgue measure from
$\gv^*$ to $\Ad^*(N)\lambda$.  The Plancherel formula on $N$ is
\begin{equation}
f(x) = c\int_{\gt^*} \Theta_{\pi_\lambda}(r_xf) |P(\lambda)|d\lambda
        \text{ for } f \in \cC(N).
\end{equation}
\end{theorem}

\section{Arithmetic Quotients}\label{quotients}
\setcounter{equation}{0}

In this section we consider the case where  our connected
simply connected nilpotent Lie group $N$ has a discrete co-compact 
subgroup $\Gamma$ that fits into a decomposition of the form
(\ref{setup}).  We show that the compact nilmanifold $N/\Gamma$ has a
corresponding foliation and derive analytic results analogous to those
of Theorem \ref{plancherel-general}.
These results include multiplicity formulae for the regular representation
of $N$ on $L^2(N/\Gamma)$. They apply in particular to the nilradicals
of minimal parabolic subgroups, as studied in Section \ref{iwasawa}.
\smallskip

Here are some basic facts about discrete uniform (i.e. co-compact) subgroups
of connected simply connected nilpotent Lie groups, mostly due to Mal\v cev.
See \cite[Chapter 2]{RAG} for an exposition.

\begin{proposition}\label{alg-1}
Let $N$ be a connected simply connected nilpotent Lie group.  Then 
the following are equivalent.
\begin{itemize}
\item $N$ has a discrete subgroup $\Gamma$ with $N/\Gamma$ compact. 
\item $N \cong N_\R$ where $N_\R$ is the group of real points in a unipotent 
  linear algebraic group defined over the rational number field $\Q$
\item $\gn$ has a basis $\{\xi_j\}$ for which 
the coefficients $c_{i,j}^k$ in 
$[\xi_i,\xi_j] = \sum c_{i,j}^k\xi_k$ are rational numbers.  
\end{itemize}
Under those conditions let $\gn_{_\Q}$ denote the rational span of 
$\{\xi_j\}$ and let $\gn_{_\Z}$ be the integral span.  Then 
$\exp(\gn_{_\Z})$ generates a discrete subgroup $N_\Z$ of $N = N_\R$ 
and $N_\R/N_\Z$ is compact.  Conversely, if $\Gamma$ is a discrete
co-compact subgroup of $N$ then the $\Z$--span of $\exp^{-1}(\Gamma)$
is a lattice in $\gn$ for which any generating set $\{\xi_j\}$ is a 
basis of $\gn$ such that the coefficients $c_{i,j}^k$ in
$[\xi_i,\xi_j] = \sum c_{i,j}^k\xi_k$ are rational numbers.
\end{proposition}

Note that the conditions of Proposition \ref{alg-1} hold for the nilpotent 
groups
studied in Section \ref{iwasawa}, where in fact one can choose the basis
$\{\xi_j\}$ of $\gn$ so that the $c_{i,j}^k$ are integers.
\smallskip

Here are the basic facts on square integrable representations in this setting,
from \cite[Theorem 7]{MW1973}:

\begin{proposition}\label{mult-1}
Let $N$ be a connected simply connected nilpotent Lie group that has
square integrable representations, and let $\Gamma$ a discrete co-compact 
subgroup.  Let $Z$ be the center of $N$ and normalize the volume form on 
$\gn / \gz$ by normalizing Haar measure on $N$ so that $N/Z\Gamma$ has 
volume $1$.  Let $P$ be the corresponding
Pfaffian polynomial on $\gz^*$.  Note that $\Gamma \cap Z$ is a lattice
in $Z$ and $\exp^{-1}(\Gamma \cap Z)$ is a lattice (denote it $\Lambda$)
in $\gz$.  That defines the dual lattice $\Lambda^*$ in $\gz^*$.  Then
a square integrable representation $\pi_\lambda$ occurs in $L^2(N/\Gamma)$
if and only if $\lambda \in \Lambda^*$, and in that case $\pi_\lambda$ 
occurs with multiplicity $|P(\lambda)|$.
\end{proposition}

\begin{definition}\label{ref-rational} 
{\rm
Let $N = N_\R$ be defined over $\Q$ as in Proposition \ref{alg-1}, so we
have a fixed rational form $N_\Q$.  We say that a connected Lie subgroup
$M \subset N$ is {\em rational} if $M \cap N_\Q$ is a rational form of
$M$, in other words if $\gm \cap \gn_{_\Q}$ contains a basis of $\gm$.
We say that a decomposition {\rm (\ref{setup})} is {\em rational} if the 
subgroups $M_r$ and $N_r$ are rational.
}
\hfill $\diamondsuit$
\end{definition}
The following is immediate from this definition.
\begin{lemma}\label{immediate}
Let $N$ be defined over $\Q$ as in {\rm Proposition \ref{alg-1}}
with rational structure defined by a discrete co-compact subgroup $\Gamma$.
If the decomposition {\rm (\ref{setup})} is rational then each 
$\Gamma \cap Z_r$ in $Z_r$\,, 
each $\Gamma \cap M_r$ in $M_r$\,, 
each $\Gamma \cap S_r$ in $S_r$\,, 
and each $\Gamma \cap N_r$ in $N_r$\,, is a
discrete co-compact subgroup defining the same rational structure as
the one defined by its intersection with $N_\Q$\,.
\end{lemma}

For the rest of this section  we will assume that $N$ and $\Gamma$ satisfy 
the rationality conditions of Lemma \ref{immediate}, in
particular that {\rm (\ref{setup})} is rational.  Then for each $r$, 
$Z_r \cap \Gamma$ is a lattice in the center $Z_r$ of $M_r$, and
$\Lambda_r := \log(Z_r \cap \Gamma)$ is a lattice in its Lie algebra $\gz_r$.
That defines the dual lattice $\Lambda_r^*$ in $\gz_r^*$.  We normalize the
Pfaffian polynomials on the $\gz_r^*$, and thus the polynomial $P$ on
$\gs^*$, by requiring that the $N_r/(S_r\cdot (N_r\cap\Gamma))$ have volume $1$.

\begin{theorem}\label{mult-form}
Let $\lambda \in \gt^*$, in other words $\lambda = \sum \lambda_r$ where
$\lambda_r \in \gz_r^*$ with $\Pf(b_{\lambda_r}) \ne 0$.  Then a
stepwise square integrable
representation $\pi_\lambda$ of $N$ occurs in $L^2(N/\Gamma)$ if and
only if each $\lambda_r \in \Lambda_r^*$\,, and 
in that case the multiplicity of $\pi_\lambda$ on $L^2(N/\Gamma)$ is
$|P(\lambda)|$.
\end{theorem}
\begin{proof} Recall $N_r = M_1 M_2 ... M_r = N_{r-1} \rtimes M_r$ semidirect 
product, where $N = M_1M_2...M_m$ and the center $Z_r$ of $M_r$
is central in $N_r$.  Fix $r \leqq m$.  By induction on dimension we assume
that Theorem \ref{mult-form} holds for $N_{r-1}$ and $\Gamma \cap N_{r-1}$.
We may also assume that $\dim Z_r = 1$, following the argument of the first
paragraph of the proof of \cite[Theorem 7]{MW1973}.  
\medskip

Now we proceed as in \cite{MW1973}, adapted to our situation.
Choose nonzero rational $x \in \gm_r \setminus \gz_r$ and $z \in \gz_r$
in such a way that (i) $\exp(z)$ generates the infinite cyclic group 
$\Gamma \cap Z_r$\,, (ii) $[x,\gm_r] \subset \gz_r$\,, and (iii) 
$\exp(x)$ and $\exp(z)$ generate $\Gamma \cap P_r$ where $P_r = \exp(\gp_r)$
where $\gp_r$ is the span of $x$ and $z$.  The centralizer $Z_{M_r}(x)$ of
of $x$ in $M_r$ is a rational normal subgroup of codimension $1$ in $M_r$,, so 
$Q_r := N_{r-1} \rtimes Z_{M_r}(x)$ is a rational normal subgroup of 
codimension $1$ in $N_r$.  The group $\Gamma Q_r / Q_r$ is infinite cyclic.
Parameterize $\gz_r^*$ by $a = a(\nu_r) = \nu_r(z)$.
The Pfaffian polynomial on $\gz_r^*$\,, normalized by the condition that
$M_r/\Gamma Z_r$ has volume $1$, satisfies $P_r(a) = \Pf(b_{\nu_r})
= c_r a^{d_r}$ where $\nu_r(z) = a$, $\dim M_r = 2d_r + 1$, and $c_r$ is 
a nonzero constant.
\medskip

Choose $\gamma \in \Gamma$ whose image in $\Gamma Q_r / Q_r$ is a generator
and let $y = \log(\gamma)$.  Then $[x,y]$ is a rational multiple of $x$,
say $[x,y] = uz$.  Since $\exp(x)$, $\exp(y)$ and $\exp(z)$ span a rational
$3$--dimensional Heisenberg algebra which we denote $\gh_r$\,; $H_r$ 
denotes the corresponding group.  
It follows \cite{A1963} that $u$ is an integer.
\medskip

Let $\pi_\nu \in \widehat{N_r}$ occur in 
$L^2(N_r/(\Gamma \cap N_r))$ where $\nu = \nu_1 + \dots + \nu_r$ with 
$\nu_i \in \gz_i$ and $\Pf(b_{\nu_i}) \ne 0$.  By induction,
$\nu_i \in \Lambda^*$ for $i < r$, and the argument immediately 
above shows that $\nu_r \in \Lambda^*$.  In other words, by induction on
dimension and on $r$, if
$\pi_\lambda$ occurs on $L^2(N/\Gamma)$ then for each index $i$
we have $\lambda_i \in \Lambda_i^*$. 
\medskip

By construction, $Q_r$ satisfies (\ref{setup}) with $Z_{M_r}(x)$ in place 
of $M_r$\,.  If $\nu_r \in \gz^*_r$ defines a square integrable (mod $Z_r$)
representation of $M_r$, then it also defines a square integrable 
(mod $P_r$) representation of $Z_{M_r}(x)$.
Let $\xi_\nu \in \widehat{Q_r}$ correspond to 
$\nu = \nu_1 + \dots + \nu_r$ with each
$\nu_i \in \Lambda^*\cap\gz_i^*$ and $\Pf(b_{\nu_i}) \ne 0$. By
induction on dimension we may assume that $\xi_\nu$ has
multiplicity $|\Pf(b_{\nu_1})\dots \Pf(b_{\nu_{r-1}})\Pf'(b_{\nu_r})|$ on 
$L^2(Q_r/(\Gamma \cap Q_r))$, where $\Pf'(b_{\nu_r})$ is the Pfaffian  
computed on the Lie algebra of $Z_{M_r}(x)$ (modulo its center $\gp_r$).
\medskip

The square integrable representations of $Z_{M_r}(x)$ are parameterized by
the linear functionals $\mu_r$ on $\gp_r = \gz_r + x\R$ with
$\Pf' \ne 0$.  We parameterize $\mu_r$ by $a = \mu_r(z)$ and $b = \mu_r(x)$ so
$\Pf'(b_{\mu_r})$ is a polynomial in $a$ and $b$.  By construction it is 
independent of $x$, so $\Pf'(b_{\mu_r}) = c'_r a^{d_r - 1}$ where
$c_r'$ is a constant and $\dim \gm_r/\gz_r = 2d_r$.
Define $\nu_r$ by  $a = \nu_r(z)$, i.e. by $\nu_r = \mu_r|_{\gz_r}$.
Since $[x,y] = uz$ now $\Pf(a) = \Pf'(a,b)au = uc_r'a^{d_r}$, in particular
$c_r = uc'_r$.
\medskip

By induction on dimension, $\xi_\mu$ occurs in $L^2(Q_r/(\Gamma \cap Q_r))$
if and only if each $\mu_i \in \Lambda^*\cap\gz_i^*$ with
$\Pf(b_{\mu_i}) \ne 0$ for $i < r$, both $a$ and $b$ are integers,
and $a \ne 0$.  To simplify the notation, fix the $\mu_i$ for $i < r$
and write $\xi(a,b)$ for the $\xi_\mu$ where $\mu_r$ has parameter $(a,b)$.
Then $\xi(a,b)$ has multiplicity $mult'(a,b) = 
|\Pf(b_{\mu_1})\dots \Pf(b_{\mu_{r-1}})\Pf'(b_{\mu_r})|
= |\Pf(b_{\mu_1})\dots \Pf(b_{\mu_{r-1}})c'_r a^{d_r - 1}|$.
Thus $c'_r$ is an integer, so $c_r = uc'_r$ is an integer as well.
\medskip

Note that $\pi_\nu = \Ind_{Q_r}^{L_r} (\xi_\mu)$ whenever $\mu|_{\gs_r} = \nu$
and that $\pi_\nu|_{Q_r}$ is the direct integral of all such $\xi_\mu$\,.
Denote $A'(\nu) = \{\mu \mid \mu|_{\gs_r} = \nu \text{ and } \xi_\mu
\text{ occurs in } L^2(Q_r/(\Gamma \cap Q_r)) \}$.
It consists of all $\xi(a,b)$ with fixed $a = \nu_r(z) \ne 0$ and 
integral $b$ if $a$ is an integer, the empty set if $a$ is not integral.
Fix a set $A(\nu)$ of representatives of the orbits of $\Gamma \cap N_r$
on $A'(\nu)$.  As in the proof of \cite[Theorem 7]{MW1973}, the algorithm 
of \cite[page 153]{M1965} says that the multiplicity 
of $\pi_\nu$ in $L^2(L_r/(\Gamma \cap L_r))$ is
$ mult(\nu) = \sum_{\mu \in A(\nu)}\, mult'(\mu).  $
\medskip

An immediate consequence: $mult(\nu) > 0$ if and only if each 
$\nu_i \in \Lambda^*$.  That proves the first assertion of the Theorem.
\medskip

We look at action of $\Gamma \cap N_r$ on $A'(\nu)$.  First,
$\Gamma \cap Q_r$ acts trivially, so the action is given by the cyclic
group $(\Gamma \cap N_r)/(\Gamma \cap Q_r)$, which has generator 
$\overline{\gamma} = \exp(y)(\Gamma \cap Q_r)$.  As $[x,y] = uz$ the action is
$\overline{\gamma}: \xi(a,b) \mapsto \xi(a,b+au)$.  So we can assume that
$A(\nu)$ consists of the $au$ elements $\xi(a,b+i)$ where $i$ is integral 
with $0 \leqq i < au$.  Each $mult'(a,b+i) = 
|\Pf(b_{\mu_1})\dots \Pf(b_{\mu_{r-1}})c'_r a^{d_r - 1}|$, so now $mult(\nu) 
= |\Pf(b_{\nu_1})\dots \Pf(b_{\nu_{r-1}})|\cdot |auc'_r a^{d_r - 1}|
= |\Pf(b_{\nu_1})\dots \Pf(b_{\nu_{r-1}})|\cdot |\Pf(b_{\nu_r})|$. 
This completes the proof of the induction step, and thus of the Theorem.
\end{proof}

\section{Appendix: Some Concrete Calculations}
\label{appendix}
\setcounter{equation}{0}
This Appendix will be left in the arXiv version of the paper but
removed before submitting the paper for publication.
\medskip

Here are concrete examples demonstrating the root structure of
the key technical tool, Lemma \ref{layers-nilpotent}, behind
Theorem \ref{iwasawa-layers}.  These are
the cases where $\Delta(\gg,\ga) = \Delta_0(\gg,\ga)$, i.e. 
where no restricted roots are divisible.  The modification to
allow divisible restricted roots is straightforward.
\medskip

First suppose that $\Delta$ is of type $A$ 
\setlength{\unitlength}{.35 mm}
\begin{picture}(120,18)
\put(5,2){\circle{2}}
\put(2,5){$\psi_1$}
\put(6,2){\line(1,0){23}}
\put(30,2){\circle{2}}
\put(27,5){$\psi_2$}
\put(31,2){\line(1,0){23}}
\put(55,2){\circle{2}}
\put(52,5){$\psi_{\ell -2}$}
\put(56,2){\line(1,0){13}}
\put(74,2){\circle*{1}}
\put(77,2){\circle*{1}}
\put(80,2){\circle*{1}}
\put(86,2){\line(1,0){13}}
\put(100,2){\circle{2}}
\put(97,5){$\psi_{\ell - 1}$}
\end{picture}
Then $r \leqq [\ell /2]$ and $\beta_r = \psi_r + \psi_{r+1} + \dots +
\psi_{\ell -r}$.  If $\alpha, \alpha' \in \Delta^+$ with $\alpha + \alpha'
= \beta_r$ then one of $\alpha, \alpha'$ must have form
$\psi_r + \dots + \psi_s$ and then the other must be $\psi_{s+1} + \dots
+ \psi_{\ell - r}$.  The assertions of Lemma \ref{layers-nilpotent}
follow by inspection.
\medskip

Now suppose that $\Delta$ is of type $B$ 
\setlength{\unitlength}{.35 mm}
\begin{picture}(120,12)
\put(5,2){\circle{2}}
\put(2,5){$\psi_1$}
\put(6,2){\line(1,0){23}}
\put(30,2){\circle{2}}
\put(27,5){$\psi_2$}
\put(31,2){\line(1,0){13}}
\put(50,2){\circle*{1}}
\put(53,2){\circle*{1}}
\put(56,2){\circle*{1}}
\put(61,2){\line(1,0){13}}
\put(75,2){\circle{2}}
\put(72,5){$\psi_{n-1}$}
\put(76,2.5){\line(1,0){23}}
\put(76,1.5){\line(1,0){23}}
\put(100,2){\circle*{2}}
\put(97,5){$\psi_n$}
\end{picture}
Then $\beta_1 = \psi_1 + 2(\psi_2 + \dots + \psi_n)$, $\beta_2 = \psi_1$,
$\beta_3 = \psi_3 + 2(\psi_4 + \dots + \psi_n)$, $\beta_4 = \psi_3$, etc.
If $r$ is even, $\beta_r = \psi_{r-1}$, then $\Delta^+_r$ is empty and the assertion
is trivial.  Now let $r$ be odd, $\beta_r = 
\psi_r + 2(\psi_{r+1} + \dots + \psi_n)$.  
If $\alpha, \alpha' \in \Delta^+$ with $\alpha + \alpha'
= \beta_r$ then one of $\alpha$ must have form $\sum a_i\psi_i$
and $\alpha'$ must have form $\sum a'_i\psi_i$ with $a_{r+1} = 1 = a'_{r+1}$,
$a_i = 0 = a'_i$ for $i < r$, and one of $a_r, a'_r$ is $1$ while the other
is $0$.  Now we may suppose
$$
\alpha = \psi_r + \psi_{r+1} + \dots + \psi_u \text{ or } 
  \alpha = \psi_r + \psi_{r+1} + \dots + \psi_u + 2(\psi_{u+1} + \dots + \psi_n)$$
and 
$$
\alpha' = \psi_{r+1} + \dots + \psi_{u'} \text{ or } 
  \alpha' = \psi_{r+1} + \dots + \psi_{u'}  + 2(\psi_{u'+1} + \dots + \psi_n)
$$
If neither of the $2$'s appears we must have 
$$
\alpha = \psi_r + \psi_{r+1} + \dots + \psi_n \text{ and }
\alpha' = \psi_{r+1} + \dots + \psi_n.
$$
The coefficient $2$ cannot occur in both $\alpha$ and $\alpha'$
because $\psi_n$ has coefficient $2$ in $\beta_r$.  
That leaves the possibilities
$$
\alpha = \psi_r + \psi_{r+1} + \dots + \psi_u \text{ and }
\alpha' = \psi_{r+1} + \dots + \psi_{u'}  + 2(\psi_{u'+1} + \dots + \psi_n)
$$
or the same with $\alpha$ and $\alpha'$ interchanged.  In these cases
$u = u' <n$.  More generally the argument shows that if $\alpha + \alpha'$
is a root then $\alpha + \alpha' = \beta_r$.

Next suppose that $\Delta$ is of type $C$ 
\setlength{\unitlength}{.5 mm}
\begin{picture}(120,12)
\put(5,2){\circle*{2}}
\put(2,5){$\psi_1$}
\put(6,2){\line(1,0){23}}
\put(30,2){\circle*{2}}
\put(27,5){$\psi_2$}
\put(31,2){\line(1,0){13}}
\put(50,2){\circle*{1}}
\put(53,2){\circle*{1}}
\put(56,2){\circle*{1}}
\put(61,2){\line(1,0){13}}
\put(75,2){\circle*{2}}
\put(72,5){$\psi_{n-1}$}
\put(76,2.5){\line(1,0){23}}
\put(76,1.5){\line(1,0){23}}
\put(100,2){\circle{2}}
\put(97,5){$\psi_n$}
\end{picture}
Then $\beta_1 = 2(\psi_1 + \dots \psi_{n-1} + \psi_n$,
$\beta_2 = 2(\psi_2 + \dots \psi_{n-1} + \psi_n$, etc.
If $\alpha, \alpha' \in \Delta^+_r$ with $\alpha + \alpha' = \beta_r$ then
one of $\alpha$ must have form $\sum a_i\psi_i$
and $\alpha'$ must have form $\sum a'_i\psi_i$ with $a_r = 1 = a'_r$,
$a_i = 0 = a'_i$ for $i < r$, and one of $a_n, a'_n$ is $1$ while the other
is $0$.  Now we may suppose
$$
\alpha = \psi_r + \psi_{r+1} + \dots + \psi_u \text{ or }
  \alpha = \psi_r + \psi_{r+1} + \dots + \psi_u + 2(\psi_{u+1} + \dots + 
      \psi_{n-1}) + \psi_n
$$
and
$$
\alpha' = \psi_r + \psi_{r+1} + \dots + \psi_{u'} \text{ or }
  \alpha' = \psi_r + \dots + \psi_{u'}  + 2(\psi_{u'+1} + \dots + 
  \psi_{n-1}) + \psi_n
$$
If the coefficient $2$ appears in neither then 
$$
\alpha = \psi_r + \psi_{r+1} + \dots + \psi_u \text{ and }
\alpha' = \psi_r + \psi_{r+1} + \dots + \psi_{u'}
$$
with one of $u, u'$ equal to $n$ and the other equal to $n-1$.  The coefficient
$2$ cannot appear in both because $\psi_n$ has coefficient $1$ in $\beta_r$.
So now, interchanging $\alpha$ and $\alpha'$ if necessary,
$$
\alpha = \psi_r + \psi_{r+1} + \dots + \psi_u \text{ and }
  \alpha' = \psi_r + \dots + \psi_{u'}  + 2(\psi_{u'+1} + \dots + 
  \psi_{n-1}) + \psi_n.
$$
For their sum to be a root we must have $u = u' < r$.
\smallskip

Now suppose that $\Delta$ is of type $D$ 
\setlength{\unitlength}{.35 mm}
\begin{picture}(110,20)
\put(5,9){\circle{2}}
\put(2,12){$\psi_1$}
\put(6,9){\line(1,0){23}}
\put(30,9){\circle{2}}
\put(27,12){$\psi_2$}
\put(31,9){\line(1,0){13}}
\put(50,9){\circle*{1}}
\put(53,9){\circle*{1}}
\put(56,9){\circle*{1}}
\put(61,9){\line(1,0){13}}
\put(75,9){\circle{2}}
\put(70,12){$\psi_{n-2}$}
\put(76,8.5){\line(2,-1){13}}
\put(90,2){\circle{2}}
\put(93,0){$\psi_n$}
\put(76,9.5){\line(2,1){13}}
\put(90,16){\circle{2}}
\put(93,14){$\psi_{n-1}$}
\end{picture}
Then $\beta_1 = \psi_1 + 2(\psi_2 + \dots + \psi_{n-2}) + \psi_{n-1} + \psi_n$,
$\beta_2 = \psi_1$, 
$\beta_3 = \psi_3 + 2(\psi_4 + \dots + \psi_{n-2}) + \psi_{n-1} + \psi_n$,
$\beta_4 = \psi_3$, etc, as for type $B$.  If $r$ is even then $\beta_r = 
\psi_{r-1}$ and $\Delta^+_r$ is empty.  Now suppose that $r$ is odd so
$\beta_r = \psi_r + 2(\psi_{r+1} + \dots + \psi_{n-2}) + \psi_{n-1} + \psi_n$.
If $\alpha, \alpha' \in \Delta^+_r$ with $\alpha + \alpha' = \beta_r$ now
$\alpha = \sum a_i\psi_i$ and $\alpha' = \sum a'_i\psi_i$ with
$a_i = 0 = a'_i$ for $i < r$, $a_{r+1} = 1 = a'_{r+1}$, and for 
$t \in \{r, n-1, n\}$ one of $a_t, a'_t$ is $0$ and the other is $1$.  
We may and do suppose that $a_n = 0$ and $a'_n = 1$.   Then 
$$
\alpha = a_r\psi_r + (\psi_{r+1} + \dots + \psi_u) \text{ with }  
	r+1 \leqq u \leqq n-1
$$
and
$$
\alpha' =  (1-a_r)\psi_r + (\psi_{r+1} + \dots + \psi_u) + 
  2(\psi_{u+1} + \dots + \psi_{n-2}) + (1-a_{n-1})\psi_{n-1} + \psi_n
$$
If $u = n-1$ then $1 - a_{n-1} = 0$ and the $\alpha'$ just above is not 
a root, so $r+1 \leqq u \leqq n-2$.  In particular $\alpha + \alpha'
\in \Delta^+$ implies $\alpha + \alpha' = \beta_r$.
\medskip

We have finished with the classical structures and go on to the 
exceptional ones.
\medskip

Suppose that $\Delta$ is of type $G_2$ 
\setlength{\unitlength}{.75 mm}
\begin{picture}(30,10)
\put(10,3){\circle{2}}
\put(8,-1){$\psi_1$}
\put(11,2){\line(1,0){13}}
\put(11,3){\line(1,0){13}}
\put(11,4){\line(1,0){13}}
\put(16,2){$\langle$}
\put(25,3){\circle{2}}
\put(23,-2){$\psi_2$}
\end{picture}.
Then $\beta_1 = 3\psi_1 + 2\psi_2$ and $\beta_2 = \psi_1$.  So
$\Delta^+_1 = \{\{3\psi_1 + \psi_2,\, \psi_2\},\, \{2\psi_1 + \psi_2,\, 
\psi_1 + \Psi_2\}\}$ in pairs $\{\alpha, \alpha'\}$ with sum $\beta_1$, 
and $\Delta^+_2$ is empty.  The assertions of Lemma \ref{layers-nilpotent}
follow by inspection.
\medskip

Suppose that $\Delta$ is of type $F_4$
\setlength{\unitlength}{.75 mm}
\begin{picture}(60,13)
\put(10,3){\circle{2}}
\put(8,-2){$\psi_1$}
\put(11,3){\line(1,0){13}}
\put(25,3){\circle{2}}
\put(23,-2){$\psi_2$}
\put(26,2.5){\line(1,0){13}}
\put(26,3.5){\line(1,0){13}}
\put(31,2){$\rangle$}
\put(40,3){\circle{2}}
\put(38,-2){$\psi_3$}
\put(41,3){\line(1,0){13}}
\put(55,3){\circle{2}}
\put(53,-2){$\psi_4$}
\end{picture}.
Then $\beta_1 = 2\psi_1 + 3\psi_2 + 4\psi_4 + 2\psi_4$,\, 
$\beta_2 = \psi_2 + 2\psi_3 + 2\psi_4$,\, $\beta_3 = \psi_2 + 2\psi_3$ and
$\beta_4 = \psi_2$. Thus (in pairs $\{\alpha, \alpha'\}$ with sum $\beta_1$)
$$
\begin{aligned}
\Delta^+_1 = \{&\{\psi_1,\, \psi_1 + 3\psi_2 + 4\psi_3 + 2\psi_4\},\,
	\{\psi_1 + \psi_2,\, \psi_1 + 2\psi_2 + 4\psi_3 + 2\psi_4\},\\
&\{\psi_1 + \psi_2 + \psi_3,\, \psi_1 + 2\psi_2 + 3\psi_3 + 2\psi_4\}, \,
\{\psi_1 + \psi_2 + 2\psi_3,\,  \psi_1 + 2\psi_2 + 2\psi_3 + 2\psi_4\}, \\
&\{\psi_1 + \psi_2 + \psi_3 + \psi_4,\, \psi_1 + 2\psi_2 + 3\psi_3 + \psi_4\},\,
\{\psi_1 + 2\psi_2 + 2\psi_3,\, \psi_1 + \psi_2 + 2\psi_3 + 2\psi_4\},\\
&\{\psi_1 + \psi_2 + 2\psi_3 + \psi_4,\,\psi_1 + 2\psi_2 + 2\psi_3 + \psi_4\}\}.
\end{aligned}
$$
Similarly,
$$
\Delta^+_2 = \{\{\psi_4,\, \psi_2 + 2\psi_3 + \psi_4\},  
      \{\psi_3 + \psi_4,\, \psi_2 + \psi_3 + \psi_4\}\},\,
\Delta^+_3 = \{\psi_3,\, \psi_2 + \psi_3\}\, \text{ and } \Delta^+_4 = \emptyset
$$
In each case if $\alpha, \alpha' \in \Delta^+_r$ and $\alpha + \alpha'$ is a root
then $\alpha + \alpha' = \beta_r$.  The assertions of Lemma 
\ref{layers-nilpotent} follow by inspection.
\medskip

Suppose that $\Delta$ is of type $E_6$ 
\setlength{\unitlength}{.35 mm}
\begin{picture}(65,23)
\put(10,12){\circle{2}}
\put(8,15){$\psi_1$}
\put(11,12){\line(1,0){13}}
\put(25,12){\circle{2}}
\put(23,15){$\psi_3$}
\put(26,12){\line(1,0){13}}
\put(40,12){\circle{2}}
\put(38,15){$\psi_4$}
\put(41,12){\line(1,0){13}}
\put(55,12){\circle{2}}
\put(53,15){$\psi_5$}
\put(56,12){\line(1,0){13}}
\put(70,12){\circle{2}}
\put(68,15){$\psi_6$}
\put(40,11){\line(0,-1){13}}
\put(40,-3){\circle{2}}
\put(43,-3){$\psi_2$}
\end{picture}.
Then the strongly orthogonal roots $\beta_i$ are given by 
$$
\begin{aligned}
\beta_1 &= \psi_1 + 2\psi_2 + 2\psi_3 + 3\psi_4 + 2\psi_5 + \psi_6,\\
\beta_2 &= \psi_1 + \psi_3 + \psi_4 + \psi_5 + \psi_6,\\ 
\beta_3 &= \psi_3 + \psi_4 + \psi_5, \text{ and }\\ 
\beta_4 &= \psi_4.  
\end{aligned}
$$
Now
$$
\begin{aligned}
\Delta^+_1 &= \{ \{\psi_2, \psi_1 + \psi_2 + 2\psi_2 + 3\psi_4 + 2\psi_5 + \psi_6\},\\
         &\{\psi_2 + \psi_4, \psi_1 + \psi_2 + 2\psi_3 + 2\psi_4 + 2\psi_5 
		+ \psi_6\},\\
         &\{\psi_2 + \psi_3 + \psi_4, \psi_1 + \psi_2 + \psi_3 + 2\psi_4 + 
		2\psi_5 + \psi_6\},\\
         &\{\psi_2 + \psi_4 + \psi_5, \psi_1 + \psi_2 + 2\psi_3 + 2\psi_4 + 
		\psi_5 + \psi_6\},\\
         &\{\psi_1 + \psi_2 + \psi_3 + \psi_4, \psi_2 + \psi_3 + 2\psi_4 + 
		2\psi_5 + \psi_6\},\\
         &\{\psi_2 + \psi_3 + \psi_4 + \psi_5, \psi_1 + \psi_2 + \psi_3 + 
		2\psi_4 + \psi_5 + \psi_6\},\\
         &\{\psi_2 + \psi_4 + \psi_5 + \psi_6, \psi_1 + \psi_2 + 2\psi_3 + 
		2\psi_4 + \psi_5\},\\
         &\{\psi_1 + \psi_2 + \psi_3 + \psi_4 + \psi_5, \psi_2 + \psi_3 + 
		2\psi_4 + \psi_5 + \psi_6\},\\
         &\{\psi_2 + \psi_3 + 2\psi_4 + \psi_5, \psi_1 + \psi_2 + \psi_3 + 
		\psi_4 + \psi_5 + \psi_6\},\\
         &\{\psi_2 + \psi_3 + \psi_4 + \psi_5 + \psi_6,  \psi_1 + \psi_2 + 
		\psi_3 + 2\psi_4 + \psi_5\} \}
\end{aligned}
$$ 
and
$$
\begin{aligned}
\Delta^+_2 = \{ &\{\psi_1, \psi_3 + \psi_4 + \psi_5 + \psi_6;
          \psi_6, \psi_1 + \psi_3 + \psi_4 + \psi_5\}; \\
          &\{\psi_1 + \psi_3, \psi_4 + \psi_5 + \psi_6,
          \psi_5 + \psi_6, \psi_1 + \psi_3 + \psi_4\} \}
\end{aligned}
$$
while
$$
\Delta^+_3 = \{\{\psi_3, \psi_4 + \psi_5\}, \{\psi_5, \psi_3 + \psi_4\}\}
\text{ and } \Delta^+_4 = \emptyset.
$$
In each case if $\alpha, \alpha' \in \Delta^+_r$ and $\alpha + \alpha'$ is a root
then $\alpha + \alpha' = \beta_r$.  The assertions of Lemma 
\ref{layers-nilpotent} follow by inspection..
\medskip

Suppose that $\Delta$ is of type $E_7$  
\setlength{\unitlength}{.35 mm}
\begin{picture}(100,23)
\put(10,12){\circle{2}}
\put(8,15){$\psi_1$}
\put(11,12){\line(1,0){13}}
\put(25,12){\circle{2}}
\put(23,15){$\psi_3$}
\put(26,12){\line(1,0){13}}
\put(40,12){\circle{2}}
\put(38,15){$\psi_4$}
\put(41,12){\line(1,0){13}}
\put(55,12){\circle{2}}
\put(53,15){$\psi_5$}
\put(56,12){\line(1,0){13}}
\put(70,12){\circle{2}}
\put(68,15){$\psi_6$}
\put(71,12){\line(1,0){13}}
\put(85,12){\circle{2}}
\put(83,15){$\psi_7$}
\put(40,11){\line(0,-1){13}}
\put(40,-3){\circle{2}}
\put(43,-3){$\psi_2$}
\end{picture}.
Then the strongly orthogonal roots $\beta_i$ are
$$
\begin{aligned}
&\beta_1 = 2\psi_1 + 2\psi_2 + 3\psi_3 + 4\psi_4 + 3\psi_5 + 2\psi_6 + \psi_7,\\
&\beta_2 = \psi_2 + \psi_3 + 2\psi_4 + 2\psi_5 + 2\psi_6 + \psi_7,\\
&\beta_3 = \psi_7,\\
&\beta_4 = \psi_2 + \psi_3 + 2\psi_4 + \psi_5,\\
&\beta_5 = \psi_2,\\
&\beta_6 = \psi_3 \text{ and } \\
&\beta_7 = \psi_5.  
\end{aligned}
$$
Now
$$
\begin{aligned}
\Delta^+_1 = \{&\{ \psi_1,\, 
	\psi_1 + 2\psi_2 + 3\psi_3 + 4\psi_4 + 3\psi_5 + 2\psi_6 + \psi_7\},\\
        &\{ \psi_1 + \psi_3,\,
	\psi_1 + 2\psi_2 + 2\psi_3 + 4\psi_4 + 3\psi_5 + 2\psi_6 + \psi_7\},\\
        &\{ \psi_1 + \psi_3 + \psi_4,\,
	 \psi_1 + 2\psi_2 + 2\psi_3 + 3\psi_4 + 3\psi_5 + 2\psi_6 + \psi_7\},\\
        &\{ \psi_1 + \psi_2 + \psi_3 + \psi_4,\,
	 \psi_1 + \psi_2 + 2\psi_3 + 3\psi_4 + 3\psi_5 + 2\psi_6 + \psi_7\},\\
        &\{ \psi_1 + \psi_3 + \psi_4 + \psi_5,\,
	 \psi_1 + 2\psi_2 + 2\psi_3 + 3\psi_4 + 2\psi_5 + 2\psi_6 + \psi_7\},\\
        &\{ \psi_1 + \psi_2 + \psi_3 + \psi_4 + \psi_5,\,
	 \psi_1 + \psi_2 + 2\psi_3 + 3\psi_4 + 2\psi_5 + 2\psi_6 + \psi_7\},\\
        &\{ \psi_1 + \psi_3 + \psi_4 + \psi_5 + \psi_6,\,
	 \psi_1 + 2\psi_2 + 2\psi_3 + 3\psi_4 + 2\psi_5 + \psi_6 + \psi_7\},\\
        &\{ \psi_1 + \psi_2 + \psi_3 + 2\psi_4 + \psi_5,\,
	 \psi_1 + \psi_2 + 2\psi_3 + 2\psi_4 + 2\psi_5 + 2\psi_6 + \psi_7\},\\
        &\{ \psi_1 + \psi_2 + \psi_3 + \psi_4 + \psi_5 + \psi_6,\,
	 \psi_1 + \psi_2 + 2\psi_3 + 3\psi_4 + 2\psi_5 + \psi_6 + \psi_7\},\\
        &\{ \psi_1 + \psi_3 + \psi_4 + \psi_5 + \psi_6 + \psi_7,\,
	 \psi_1 + 2\psi_2 + 2\psi_3 + 3\psi_4 + 2\psi_5 + \psi_6\},\\
        &\{ \psi_1 + \psi_2 + 2\psi_3 + 2\psi_4 + \psi_5,\,
	 \psi_1 + \psi_2 + \psi_3 + 2\psi_4 + 2\psi_5 + 2\psi_6 + \psi_7\},\\
        &\{ \psi_1 + \psi_2 + \psi_3 + 2\psi_4 + \psi_5 + \psi_6,\,
	 \psi_1 + \psi_2 + 2\psi_3 + 2\psi_4 + 2\psi_5 + \psi_6 + \psi_7\},\\
        &\{ \psi_1 + \psi_2 + \psi_3 + \psi_4 + \psi_5 + \psi_6 + \psi_7,\,
	 \psi_1 + \psi_2 + 2\psi_3 + 3\psi_4 + 2\psi_5 + \psi_6\},\\
        &\{ \psi_1 + \psi_2 + 2\psi_3 + 2\psi_4 + \psi_5 + \psi_6,\,
	 \psi_1 + \psi_2 + \psi_3 + 2\psi_4 + 2\psi_5 + \psi_6 + \psi_7\},\\
        &\{ \psi_1 + \psi_2 + \psi_3 + 2\psi_4 + 2\psi_5 + \psi_6,\,
	 \psi_1 + \psi_2 + 2\psi_3 + 2\psi_4 + \psi_5 + \psi_6 + \psi_7\},\\
        &\{ \psi_1 + \psi_2 + \psi_3 + 2\psi_4 + \psi_5 + \psi_6 + \psi_7,\,
	 \psi_1 + \psi_2 + 2\psi_3 + 2\psi_4 + 2\psi_5 + \psi_6\}\},
\end{aligned}
$$
and
$$
\begin{aligned}
\Delta^+_2 = \{&\{\psi_6,
	     \psi_2 + \psi_3 + 2\psi_4 + 2\psi_5 + \psi_6 + \psi_7\},\,
        \{ \psi_5 + \psi_6,
	     \psi_2 + \psi_3 + 2\psi_4 + \psi_5 + \psi_6 + \psi_7\},\\
        &\{\{psi_6 + \psi_7,
	     \psi_2 + \psi_3 + 2\psi_4 + 2\psi_5 + \psi_6;\},\,
        \{ \psi_4 + \psi_5 + \psi_6,
	     \psi_2 + \psi_3 + \psi_4 + \psi_5 + \psi_6 + \psi_7\}, \\
        &\{\psi_5 + \psi_6 + \psi_7,
	     \psi_2 + \psi_3 + 2\psi_4 + \psi_5 + \psi_6;\},\,
        \{ \psi_2 + \psi_4 + \psi_5 + \psi_6,
	     \psi_3 + \psi_4 + \psi_5 + \psi_6 + \psi_7\},\\
        &\{\psi_3 + \psi_4 + \psi_5 + \psi_6,
	     \psi_2 + \psi_4 + \psi_5 + \psi_6 + \psi_7\},\,
        \{ \psi_4 + \psi_5 + \psi_6 + \psi_7,
	     \psi_2 + \psi_3 + \psi_4 + \psi_5 + \psi_6\} \}
\end{aligned}
$$
while
$$
\begin{aligned}
\Delta^+_4 = \{&\{\psi_4, \psi_2 + \psi_3 + \psi_4 + \psi_5\},\,
        \{\psi_2 + \psi_4, \psi_3 + \psi_4 + \psi_5\},\\
        &\{\psi_3 + \psi_4, \psi_2 + \psi_4 + \psi_5\},\,
        \{\psi_4 + \psi_5, \psi_2 + \psi_3 + \psi_4\}\}
\end{aligned}
$$
and 
$$
\Delta^+_3 = \Delta^+_5 = \Delta^+_6 = \Delta^+_7 = \emptyset.
$$
In each case if $\alpha, \alpha' \in \Delta^+_r$ and $\alpha + \alpha'$ is a root
then $\alpha + \alpha' = \beta_r$.  The assertions of Lemma 
\ref{layers-nilpotent} follow by inspection.
\medskip

Finally suppose that $\Delta$ is of type $E_8$
\setlength{\unitlength}{.35 mm}
\begin{picture}(110,23)
\put(10,12){\circle{2}}
\put(8,15){$\psi_1$}
\put(11,12){\line(1,0){13}}
\put(25,12){\circle{2}}
\put(23,15){$\psi_3$}
\put(26,12){\line(1,0){13}}
\put(40,12){\circle{2}}
\put(38,15){$\psi_4$}
\put(41,12){\line(1,0){13}}
\put(55,12){\circle{2}}
\put(53,15){$\psi_5$}
\put(56,12){\line(1,0){13}}
\put(70,12){\circle{2}}
\put(68,15){$\psi_6$}
\put(71,12){\line(1,0){13}}
\put(85,12){\circle{2}}
\put(83,15){$\psi_7$}
\put(86,12){\line(1,0){13}}
\put(100,12){\circle{2}}
\put(98,15){$\psi_8$}
\put(40,11){\line(0,-1){13}}
\put(40,-3){\circle{2}}
\put(43,-3){$\psi_2$}
\end{picture}.
Then 
$$
\begin{aligned}
&\beta_1 = 2\psi_1 + 3\psi_2 + 4\psi_3 + 6\psi_4 + 5\psi_5
	+ 4\psi_6 + 3\psi_7 + 2\psi_8,\\
&\beta_2 = 2\psi_1 + 2\psi_2 + 3\psi_3 + 4\psi_4 + 3\psi_5 + 2\psi_6 + \psi_7,\\
&\beta_3 = \psi_2 + \psi_3 + 2\psi_4 + 2\psi_5 + 2\psi_6 + \psi_7,\\
&\beta_4 = \psi_7,\\
&\beta_5 = \psi_2 + \psi_3 + 2\psi_4 + \psi_5,\\
&\beta_6 = \psi_2, \\ 
&\beta_7 = \psi_3 \text{ and }\\
&\beta_8 = \psi_5.
\end{aligned}
$$
Now
$$
\begin{aligned}
\Delta^+_1 = \{&\{\psi_8 ,\phantom{X}
		2\psi_1+3\psi_2+4\psi_3+6\psi_4+5\psi_5+4\psi_6+3\psi_7+\psi_8\},\\
        &\{\psi_7+\psi_8 ,\phantom{X}
		2\psi_1+3\psi_2+4\psi_3+6\psi_4+5\psi_5+4\psi_6+2\psi_7+\psi_8\},\\
        &\{\psi_6+\psi_7+\psi_8 ,\phantom{X}
		2\psi_1+3\psi_2+4\psi_3+6\psi_4+5\psi_5+3\psi_6+2\psi_7+\psi_8\},\\
	&\{\psi_5+\psi_6+\psi_7+\psi_8 ,\phantom{X}
		2\psi_1+3\psi_2+4\psi_3+6\psi_4+4\psi_5+3\psi_6+2\psi_7+\psi_8\},\\
	&\{\psi_4+\psi_5+\psi_6+\psi_7+\psi_8 ,\phantom{X}
		2\psi_1+3\psi_2+4\psi_3+5\psi_4+4\psi_5+3\psi_6+2\psi_7+\psi_8\},\\
        &\{\psi_2+\psi_4+\psi_5+\psi_6+\psi_7+\psi_8 ,\phantom{X}
		2\psi_1+2\psi_2+4\psi_3+5\psi_4+4\psi_5+3\psi_6+2\psi_7+\psi_8\},\\
        &\{\psi_3+\psi_4+\psi_5+\psi_6+\psi_7+\psi_8 ,\phantom{X}
		2\psi_1+3\psi_2+3\psi_3+5\psi_4+4\psi_5+3\psi_6+2\psi_7+\psi_8\},\\
        &\{\psi_1+\psi_3+\psi_4+\psi_5+\psi_6+\psi_7+\psi_8 ,\phantom{X}
		\psi_1+3\psi_2+3\psi_3+5\psi_4+4\psi_5+3\psi_6+2\psi_7+\psi_8\},\\
        &\{\psi_2+\psi_3+\psi_4+\psi_5+\psi_6+\psi_7+\psi_8 ,\phantom{X}
		2\psi_1+2\psi_2+3\psi_3+5\psi_4+4\psi_5+3\psi_6+2\psi_7+\psi_8\},\\
        &\{\psi_1+\psi_2+\psi_3+\psi_4+\psi_5+\psi_6+\psi_7+\psi_8 ,\phantom{X}
		\psi_1+2\psi_2+3\psi_3+5\psi_4+4\psi_5+3\psi_6+2\psi_7+\psi_8\},\\
	&\{\psi_2+\psi_3+2\psi_4+\psi_5+\psi_6+\psi_7+\psi_8 ,\phantom{X}
		2\psi_1+2\psi_2+3\psi_3+4\psi_4+4\psi_5+3\psi_6+2\psi_7+\psi_8\},\\
	&\{\psi_1+\psi_2+\psi_3+2\psi_4+\psi_5+\psi_6+\psi_7+\psi_8 ,\phantom{X}
		\psi_1+2\psi_2+3\psi_3+4\psi_4+4\psi_5+3\psi_6+2\psi_7+\psi_8\},\\
	&\{\psi_2+\psi_3+2\psi_4+2\psi_5+\psi_6+\psi_7+\psi_8 ,\phantom{X}
		2\psi_1+2\psi_2+3\psi_3+4\psi_4+3\psi_5+3\psi_6+2\psi_7+\psi_8\},\\
	&\{\psi_1+\psi_2+2\psi_3+2\psi_4+\psi_5+\psi_6+\psi_7+\psi_8 ,\phantom{X}
		\psi_1+2\psi_2+2\psi_3+4\psi_4+4\psi_5+3\psi_6+2\psi_7+\psi_8\},\\
	&\{\psi_1+\psi_2+\psi_3+2\psi_4+2\psi_5+\psi_6+\psi_7+\psi_8 ,\phantom{X}
		\psi_1+2\psi_2+3\psi_3+4\psi_4+3\psi_5+3\psi_6+2\psi_7+\psi_8\},\\
	&\{\psi_2+\psi_3+2\psi_4+2\psi_5+2\psi_6+\psi_7+\psi_8 ,\phantom{X}
		2\psi_1+2\psi_2+3\psi_3+4\psi_4+3\psi_5+2\psi_6+2\psi_7+\psi_8\},\\
	&\{\psi_1+\psi_2+2\psi_3+2\psi_4+2\psi_5+\psi_6+\psi_7+\psi_8 ,\phantom{X}
		\psi_1+2\psi_2+2\psi_3+4\psi_4+3\psi_5+3\psi_6+2\psi_7+\psi_8\},\\
	&\{\psi_1+\psi_2+\psi_3+2\psi_4+2\psi_5+2\psi_6+\psi_7+\psi_8 ,\phantom{X}
		\psi_1+2\psi_2+3\psi_3+4\psi_4+3\psi_5+2\psi_6+2\psi_7+\psi_8\},\\
	&\{\psi_2+\psi_3+2\psi_4+2\psi_5+2\psi_6+2\psi_7+\psi_8 ,\phantom{X}
		2\psi_1+2\psi_2+3\psi_3+4\psi_4+3\psi_5+2\psi_6+\psi_7+\psi_8\},\\
	&\{ \psi_1+\psi_2+2\psi_3+3\psi_4+2\psi_5+\psi_6+\psi_7+\psi_8 ,\phantom{X}
		\psi_1+2\psi_2+2\psi_3+3\psi_4+3\psi_5+3\psi_6+2\psi_7+\psi_8\},\\
	&\{ \psi_1+\psi_2+2\psi_3+2\psi_4+2\psi_5+2\psi_6+\psi_7+\psi_8 ,\phantom{X}
		\psi_1+2\psi_2+2\psi_3+4\psi_4+3\psi_5+2\psi_6+2\psi_7+\psi_8\},\\
	&\{ \psi_1+\psi_2+\psi_3+2\psi_4+2\psi_5+2\psi_6+2\psi_7+\psi_8 ,\phantom{X}
		\psi_1+2\psi_2+3\psi_3+4\psi_4+3\psi_5+2\psi_6+\psi_7+\psi_8\},\\
	&\{ \psi_1+2\psi_2+2\psi_3+3\psi_4+2\psi_5+\psi_6+\psi_7+\psi_8 ,\phantom{X}
		\psi_1+\psi_2+2\psi_3+3\psi_4+3\psi_5+3\psi_6+2\psi_7+\psi_8\},\\
	&\{ \psi_1+\psi_2+2\psi_3+3\psi_4+2\psi_5+2\psi_6+\psi_7+\psi_8 ,\phantom{X}
		\psi_1+2\psi_2+2\psi_3+3\psi_4+3\psi_5+2\psi_6+2\psi_7+\psi_8\},\\
	&\{ \psi_1+\psi_2+2\psi_3+2\psi_4+2\psi_5+2\psi_6+2\psi_7+\psi_8 ,\phantom{X}
		\psi_1+2\psi_2+2\psi_3+4\psi_4+3\psi_5+2\psi_6+\psi_7+\psi_8\},\\
	&\{ \psi_1+2\psi_2+2\psi_3+3\psi_4+2\psi_5+2\psi_6+\psi_7+\psi_8 ,\phantom{X}
		\psi_1+\psi_2+2\psi_3+3\psi_4+3\psi_5+2\psi_6+2\psi_7+\psi_8\},\\
	&\{ \psi_1+\psi_2+2\psi_3+3\psi_4+3\psi_5+2\psi_6+\psi_7+\psi_8 ,\phantom{X}
		\psi_1+2\psi_2+2\psi_3+3\psi_4+2\psi_5+2\psi_6+2\psi_7+\psi_8\},\\
	&\{ \psi_1+\psi_2+2\psi_3+3\psi_4+2\psi_5+2\psi_6+2\psi_7+\psi_8 ,\phantom{X}
		\psi_1+2\psi_2+2\psi_3+3\psi_4+3\psi_5+2\psi_6+\psi_7+\psi_8 \}\}
\end{aligned}
$$
while
$$
\begin{aligned}
\Delta^+_2 = \{&\{ \psi_1, \phantom{X}
		\psi_1+2\psi_2+3\psi_3+4\psi_4+3\psi_5+2\psi_6+\psi_7\},\\
	&\{ \psi_1+\psi_3, \phantom{X}
		\psi_1+2\psi_2+2\psi_3+4\psi_4+3\psi_5+2\psi_6+\psi_7\},\\
	&\{ \psi_1+\psi_3+\psi_4, \phantom{X}
		\psi_1+2\psi_2+2\psi_3+3\psi_4+3\psi_5+2\psi_6+\psi_7\},\\
	&\{ \psi_1+\psi_2+\psi_3+\psi_4, \phantom{X}
		\psi_1+\psi_2+2\psi_3+3\psi_4+3\psi_5+2\psi_6+\psi_7\},\\
	&\{ \psi_1+\psi_3+\psi_4+\psi_5, \phantom{X}
		\psi_1+2\psi_2+2\psi_3+3\psi_4+2\psi_5+2\psi_6+\psi_7\},\\
	&\{ \psi_1+\psi_2+\psi_3+\psi_4+\psi_5, \phantom{X}
		\psi_1+\psi_2+2\psi_3+3\psi_4+2\psi_5+2\psi_6+\psi_7\},\\
	&\{ \psi_1+\psi_3+\psi_4+\psi_5+\psi_6, \phantom{X}
		\psi_1+2\psi_2+2\psi_3+3\psi_4+2\psi_5+\psi_6+\psi_7\},\\
	&\{ \psi_1+\psi_2+\psi_3+2\psi_4+\psi_5, \phantom{X}
		\psi_1+\psi_2+2\psi_3+2\psi_4+2\psi_5+2\psi_6+\psi_7\},\\
	&\{ \psi_1+\psi_2+\psi_3+\psi_4+\psi_5+\psi_6, \phantom{X}
		\psi_1+\psi_2+2\psi_3+3\psi_4+2\psi_5+\psi_6+\psi_7\},\\
	&\{ \psi_1+\psi_3+\psi_4+\psi_5+\psi_6+\psi_7, \phantom{X}
		\psi_1+2\psi_2+2\psi_3+3\psi_4+2\psi_5+\psi_6\},\\
	&\{ \psi_1+\psi_2+2\psi_3+2\psi_4+\psi_5, \phantom{X}
		\psi_1+\psi_2+\psi_3+2\psi_4+2\psi_5+2\psi_6+\psi_7\},\\
	&\{ \psi_1+\psi_2+\psi_3+2\psi_4+\psi_5+\psi_6, \phantom{X}
		\psi_1+\psi_2+2\psi_3+2\psi_4+2\psi_5+\psi_6+\psi_7\},\\
	&\{ \psi_1+\psi_2+\psi_3+\psi_4+\psi_5+\psi_6+\psi_7, \phantom{X}
		\psi_1+\psi_2+2\psi_3+3\psi_4+2\psi_5+\psi_6\},\\
	&\{ \psi_1+\psi_2+2\psi_3+2\psi_4+\psi_5+\psi_6, \phantom{X}
		\psi_1+\psi_2+\psi_3+2\psi_4+2\psi_5+\psi_6+\psi_7\},\\
	&\{ \psi_1+\psi_2+\psi_3+2\psi_4+2\psi_5+\psi_6, \phantom{X}
		\psi_1+\psi_2+2\psi_3+2\psi_4+\psi_5+\psi_6+\psi_7\},\\
	&\{ \psi_1+\psi_2+\psi_3+2\psi_4+\psi_5+\psi_6+\psi_7, \phantom{X}
		\psi_1+\psi_2+2\psi_3+2\psi_4+2\psi_5+\psi_6 \}\},
\end{aligned}
$$
and
$$
\begin{aligned}
\Delta^+_3 = \{&\{ \psi_6,\,   \psi_2+\psi_3+2\psi_4+2\psi_5+\psi_6+\psi_7 \}, \\
  &\{ \psi_5+\psi_6 ,\,  \psi_2+\psi_3+2\psi_4+\psi_5+\psi_6+\psi_7 \}, \\
  &\{ \psi_6+\psi_7 ,\,  \psi_2+\psi_3+2\psi_4+2\psi_5+\psi_6 \}, \\
  &\{ \psi_4+\psi_5+\psi_6 ,\,  \psi_2+\psi_3+\psi_4+\psi_5+\psi_6+\psi_7 \}, \\
  &\{ \psi_5+\psi_6+\psi_7 ,\,  \psi_2+\psi_3+2\psi_4+\psi_5+\psi_6 \}, \\
  &\{ \psi_2+\psi_4+\psi_5+\psi_6 ,\,  \psi_3+\psi_4+\psi_5+\psi_6+\psi_7 \}, \\
  &\{ \psi_3+\psi_4+\psi_5+\psi_6 ,\,  \psi_2+\psi_4+\psi_5+\psi_6+\psi_7 \}, \\
  &\{ \psi_4+\psi_5+\psi_6+\psi_7 ,\,  \psi_2+\psi_3+\psi_4+\psi_5+\psi_6 \}\},
\end{aligned}
$$
and
$$
\begin{aligned}
\Delta^+_5 = \{&\{ \psi_4,\,           \psi_2+\psi_3+\psi_4+\psi_5\}, \\
        &\{ \psi_2+\psi_4,\,    \psi_3+\psi_4+\psi_5\}, \\
        &\{ \psi_3+\psi_4,\,    \psi_2+\psi_4+\psi_5\}, \\
        &\{ \psi_4+\psi_5,\,    \psi_2+\psi_3+\psi_4\} \},
\end{aligned}
$$
while
$$
\Delta^+_4 = \Delta^+_6 = \Delta^+_7 = \Delta^+_8 = \emptyset .
$$
This completes a direct computational verification of
the assertions of Lemma \ref{layers-nilpotent}, as opposed to the
more direct structural argument in Section \ref{iwasawa}.


\begin{thebibliography}{XX}

\bibitem{A1963} L. Auslander et al, ``Flows on Homogeneous Spaces'',
Ann. Math. Studies {\bf 53}, 1963.

\bibitem{BD2001} M. L. Barberis \& I. Dotti, 
Abelian hypercomplex structures on central extensions of H-type Lie
algebras, J. Pure Appl. Algebra {\bf 158} (2001), 15–23.

\bibitem{BD2004} M. L. Barberis \& I. Dotti,  Abelian complex structures on solvable Lie
algebras, J. Lie Theory {\bf 14} (2004), 25--34.

\bibitem{BD2011} M. L. Barberis \& I. Dotti, private communication.

\bibitem{Ka} A. Kaplan,  Riemannian nilmanifolds attached to Clifford
modules, Geom. Dedicata {\bf 11} (1981), 127--136.

\bibitem{K1962} A. A. Kirillov,
Unitary representations of nilpotent Lie groups,
Uspekhi Math. Nauk {\bf 17} (1962), 57--110
(English: Russian Math. Surveys {\bf 17} (1962), 53--104).

\bibitem{M1965} C. C. Moore,
Decomposition of unitary representations defined by discrete subgroups
of nilpotent groups, Ann. Math. {\bf 82} (1965), 146--182

\bibitem{MW1973} C. C. Moore \& J. A. Wolf,
Square integrable representations of nilpotent groups.
Transactions of the American Mathematical Society,
{\bf 185} (1973), 445--462.

\bibitem{P1967} L. Puk\' anszky,
On characters and the Plancherel formula of nilpotent groups,
J. Functional Analysis {\bf 1} (1967), 255--280.

\bibitem{RAG} M. S. Raghunathan,
``Discrete Subgroups of Lie Groups'', Ergebnisse der Mathematik
und ihrer Grenzgebeite {\bf 68}, 1972.

\bibitem{W1979} J. A. Wolf,
Classification and Fourier inversion for parabolic subgroups with square
integrable nilradical.   Memoirs of the American Mathematical Society,
Number 225, 1979.

\bibitem{W2007} J. A. Wolf,
 Harmonic Analysis on Commutative Spaces.  Math. Surveys \&
Monographs, vol. 142, American Mathematical Society, 2007.

\end{thebibliography}
\enddocument
\end